\documentclass[11pt]{article}

\usepackage[top=30truemm, bottom=30truemm, left=25truemm, right=25truemm]{geometry}

\usepackage{amsmath}
\usepackage{amssymb}
\usepackage{array}
\usepackage{bm}
\usepackage{amscd}
\usepackage{amsthm}

\newtheorem{thm}{Theorem}[section]
\newtheorem{dfn}[thm]{Definition}
\newtheorem{prop}[thm]{Proposition}
\newtheorem{cor}[thm]{Corollary}
\newtheorem{lem}[thm]{Lemma}
\newtheorem{rem}[thm]{Remark}
\newtheorem{clm}[thm]{Claim}
\theoremstyle{remark}

\numberwithin{equation}{section}

\allowdisplaybreaks[4]

\newcommand{\ind}{\mathrm{ind}}
\newcommand{\knl}{\mathrm{Ker}}
\newcommand{\homo}{\mathrm{Hom}}
\newcommand{\res}{\mathrm{Res}}

\usepackage[dvipdfmx, bookmarkstype=toc,colorlinks=true, urlcolor=black, linkcolor=black, citecolor=black, linktocpage=true, bookmarks=false]{hyperref}

\begin{document}

\title{An explicit formula for the transversal indices of the lifted Dolbeault operators}
\author{Naoto Miseki}
\date{}
\maketitle

\begin{abstract}
In \cite{Atiyahtrans}, Atiyah proved that the index of a transversally elliptic operator relative to a free action can be computed by using indices of elliptic operators on the orbit manifold.
In this paper, we derive an explicit formula for the transversal indices on $S^1$-bundles over $\mathbb{CP}^n$.
Using this explicit formula and the index map, we prove Lefschetz formula for $\mathbb{CP}^n$ with the canonical action of $T^{n + 1}$.
\end{abstract}

\tableofcontents

\section{Introduction}

For elliptic pseudo-differential operators, Atiyah--Singer index theorem allows us to compute the analytic indices of them by using the topological formula. When $A$ is a transversally elliptic operator relative to an action of a compact Lie group $G$, there are also cohomological formulae, which are given by Berline and Vergne \cite{Vergne96}, Paradan and Vergne \cite{Vergne08}. However these formulae do not allow us to compute indices of transversally elliptic operators easily.
This paper gives explicit expressions of the analytic transversal indices of concrete examples and their application.

In general, it is not easy to compute transversal indices, however operators which are transversally elliptic relative to free actions can be reduced to elliptic operators on the orbit manifolds.
Let us briefly explain that property.
Let $G$ and $H$ be two compact Lie groups and $P$ be a compact manifold on which $G \times H$ acts smoothly.
We assume that the action of $H$ is free. Then $M := P/H$ is a compact manifold which is provided with an action of $G$ and the quotient map $p \colon P \to M$ is a principal $H$-bundle and $G$-equivariant.
Let
\[
  A \colon \Gamma(M,\, \mathcal{E}^{+}) \to \Gamma(M,\, \mathcal{E}^-)
\]
be a $G$-invariant elliptic operator on $M$. If $P$ is trivial, i.e., $P \cong M \times H$, $A$ lifts to $\tilde{A}$ on $P$ naturally. In the general case, using local trivializations of $P$ and a partition of unity on $M$, and averaging on $G$, $A$ on $M$ lifts to 
\[
  \tilde{A} \colon \Gamma(P,\, p^*\mathcal{E}^{+}) \to \Gamma(P,\, p^*\mathcal{E}^-) ,
\]
which is a $G \times H$-invariant $H$-transversally elliptic operator on $P$ since $A$ is elliptic on $M$.
Conversely, at the level of $K$-groups, $G \times H$-invariant $H$-transversally elliptic operators on $P$ can be represented by lifts of $G$-invariant elliptic operators on $M$. See \ref{freeaction} for details.

We denote by $\mathcal{D}'(G \times H)$ the space of distributions on $G \times H$.
In \cite{Atiyahtrans}, Atiyah proved that the transversal index of the lifted operator $\tilde{A}$ is given by
\[
  \ind_{G \times H} \tilde{A} = \sum_{\tau \in \hat{H}} \ind_G (A \otimes \underline{V_{\tau^*}} ) \chi_{\tau}  
    \quad \in \mathcal{D}'(G \times H) ,
\]
where $\hat{H}$ is the set of equivalence classes of finite dimensional irreducible (complex) representations of $H$, $\chi_{\tau}$ is the character of a representation $(\tau,\, V_{\tau}) \in \hat{H}$. $\underline{V_{\tau^*}}$ is a complex vector bundle over $M$ defined by $P \underset{H}{\times} V_{\tau^*}$ and $A \otimes \underline{V_{\tau^*}}$ denotes the operator whose principal symbol is
\[
  \sigma(A) \otimes \rm{Id_{\underline{V_{\tau^*}}}} \colon \pi^*(\mathcal{E}^+ \otimes \underline{V_{\tau^*}}) 
    \to \pi^* (\mathcal{E}^- \otimes \underline{V_{\tau^*}}) ,
\]
which is $G$-invariant elliptic, where $\sigma(A)$ is the principal symbol of $A$, and $\pi \colon T^*M \to M$ is the natural projection.
In this way, the index of a transversally elliptic operator relative to a free action can be computed by using indices of elliptic operators on the orbit manifold.

Let $\mathcal{O}(-1)$ be the tautological line bundle over $\mathbb{CP}^n$ and $\mathcal{O}(k) := \mathcal{O}(-1)^{\otimes -k} \, (k \in \mathbb{Z})$. 
We denote by $S(\mathcal{O}(k))$ the circle bundle of $\mathcal{O}(k)$ with respect to a certain Hermitian metric.
Then $p_k \colon S(\mathcal{O}(k)) \to \mathbb{CP}^n$ is a principal $S^1$-bundle. We give an action of $T^{n+1}$ on $S(\mathcal{O}(k))$ which is the lift of the canonical action of $T^{n+1}$ on $\mathbb{CP}^n$:
\begin{align*}
T^{n+1} \ni (u_1,\, \dots ,\, u_{n+1}) \colon [z_1 : \dots : z_{n+1}] \mapsto [u_1z_1 : \dots : u_{n+1}z_{n+1}] .
\end{align*}
We consider the Dirac operator $D_n := \bar{\, \partial} + \bar{\, \partial}^*$ on $\mathbb{CP}^n$ associated with the Dolbeault operator, which is a $T^{n+1}$-invariant elliptic operator, and defines an element $[ D_n ]$ of $K_{T^{n+1}} (T^*\mathbb{CP}^n)$.
As explained above, the Dirac operator $D_n$ lifts to a $T^{n+1} \times S^1$-invariant $S^1$-transversally elliptic operator on $S(\mathcal{O}(k))$ which defines an element of $K_{T^{n+1} \times S^1} (T_{S^1}^* S(\mathcal{O}(k)))$. We denote the element by $p_k^* [ D_n ]$. Here $T_{S^1}^*S(\mathcal{O}(k))$ is a subset of cotangent vectors which vanish on tangent vectors to the $S^1$-orbits. 
Our main theorem is the following:
\begin{thm} \label{mainthm}
  The transversal index of $p_k^* [ D_n ] \in K_{T^{n+1} \times S^1} (T_{S^1}^* S(\mathcal{O}(k)))$ is given by
    \begin{equation} \label{main}
      \ind_{T^{n+1} \times S^1} \left( p_k^* [ D_n ] \right) = \left( \sum_{l=0}^n \left( \sum_{i=1}^{n+1-l} (-1)^{i+l+1} 
        \chi_{n,\, ik} \, s_{i+l}(t_1^k,\,\dots, \, t_{n+1}^k)\right) t^l \right)J(\lambda_{n+1}(k))
    \end{equation}
  in $\mathcal{D}'(T^{n+1} \times S^1)$ for $n \geq 1$,\, $k \in \mathbb{Z}$.
\end{thm}
$t_1,\, \dots ,\, t_{n+1}$ and $t$ are the generators of the representation rings $R(T^{n+1})$ and $R(S^1)$ respectively, $\chi_{n,\, ik}$ is the character of the representation of $T^{n+1}$ (see (\ref{defchi}), Corollary \ref{chicharacter}), $s_{i+l}(t_1^k,\, \dots ,\, t_{n+1}^k)$ is the $(i+l)$-th elementary symmetric polynomial in variables $t_1^k,\, \dots ,\, t_{n+1}^k$.
In addition, $\displaystyle \lambda_{n+1}(k) := \prod_{j = 1}^{n+1} (1 - t_j^k t)$, and
$J(\lambda_{n+1}(k))$ is the distribution on $T^{n+1} \times S^1$ defined by
\[
J(\lambda_{n+1}(k)) :=  \prod_{j=1}^{n+1} \left[ \frac{1}{1-t_j^kt} \right]^{t = 0} - \prod_{j=1}^{n+1} \left[ \frac{1}{1-t_j^kt} \right]^{t = \infty} ,
\]
where $\left[ \frac{1}{1-t_j^kt} \right]^{t = 0}$ and $\left[ \frac{1}{1-t_j^kt} \right]^{t = \infty}$ are the Laurent expansions of $\frac{1}{1-t_j^kt}$ at $t = 0$ and $t = \infty$ respectively.
For example, 
  \begin{align*}
    &\left[ \frac{1}{1-t} \right]^{t = 0} = 1 + t + t^2 + t^3 + \cdots \\
    &\left[ \frac{1}{1-t} \right]^{t = \infty} = - t^{-1} - t^{-2} - t^{-3} - \cdots
  \end{align*}
then, we have
\[
  \left[ \frac{1}{1-t} \right]^{t = 0} - \left[ \frac{1}{1-t} \right]^{t = \infty} = \sum_{m \in \mathbb{Z}} t^m .
\]
This is the Fourier expansion of the delta function on $S^1$ supported at $1 \in S^1$. See \ref{formula} for details.

In the last section, we give an application of our main theorem, in which we prove Lefschetz formula for $\mathbb{CP}^n$ with the canonical action of $T^{n+1}$ as below.

\begin{cor} \label{application}
  For any $\eta = [ E^0 ] - [ E^1 ] \in K_{T^{n+1}}(\mathbb{CP}^n)$,  
    \begin{equation} \label{app}
      \ind_{T^{n+1}} \left( [ D_n ] \otimes \eta \right) = 
        - \sum_{t = \xi \in S^1} \res \left( \frac{\alpha ( \eta )}{\lambda_{n+1}(-1)} \frac{1}{t} \; \; ; \; \xi \right) 
          \quad \in R(T^{n+1}) .
    \end{equation}
  where $\alpha$ is the ring isomorphism from $K_{T^{n+1}}(\mathbb{CP}^n)$ to 
    $R(T^{n+1} \times S^1) / ( \lambda_{n+1}(-1) )$ (see Remark \ref{beta} for details).

In paticular, at $(t_1,\, \dots ,\, t_{n+1}) \in T^{n+1}$ such that $t_i \neq t_j (i \neq j)$, 
  \begin{equation} \label{lefforps}
    \ind_{T^{n+1}} \left( [ D_n ] \otimes \eta \right) \left( t_1, \, \cdots ,\, t_{n+1} \right) = 
      \sum_{i = 1}^{n+1} \frac{ \chi ( E^0|_{p_i} ) - \chi ( E^1|_{p_i} )}{\displaystyle 
        \prod_{j \neq i} \left( 1 - t_j^{-1} t_i \right)} \quad \in \mathbb{C} ,
  \end{equation}
where $p_i = [ 0 : \dots : 0 : \overset{i}{ \breve{1} } : 0 :  \dots : 0 ] \in \mathbb{CP}^n$ $( 1 \leq i \leq n+1)$ are the fixed points of the action of $T^{n + 1}$ and 
$\chi \left( E^j |_{p_i} \right)$ is the character of the representation space $E^j|_{p_i}$ of $T^{n + 1}$. 
\end{cor}
Note that the assumption $t_i \neq t_j$ for $i \neq j$ is equivalent to the condition that the fixed point set of $( t_1 ,\, \dots ,\, t_{n + 1} ) \in T^{n + 1}$ is a finite subset of $\mathbb{CP}^n$. 
The right hand side in (\ref{lefforps}) makes sense at $(t_1,\, \dots ,\, t_{n+1}) \in T^{n+1}$ such that $t_i \neq t_j (i \neq j)$, however the right hand side in (\ref{app}) makes sense at all points $(t_1,\, \dots ,\, t_{n+1}) \in T^{n+1}$ even if its action on $\mathbb{CP}^n$ has infinitely many fixed points.
The deformations proved in Proposition \ref{B} and Lemma \ref{residue} have important roles in the proof of this application.

Finally, let us briefly review the proof of the main theorem.
As explained above, the transversal index of $p_k^* [ D_n ]$ can be reduced to the indices of the elliptic operators:
\[
  \ind_{T^{n+1}} ( [ D_n ] \otimes \underline{t}^{-m}) \quad (m \in \mathbb{Z}) ,
\]
where $\underline{t}^{-m}$ is a holomorphic line bundle over $\mathbb{CP}^n$ (see \ref{dolop}).
These indices can be computed by Lefschetz formula, however in this paper, we compute them by using the fundamental tools in Complex Geometry: Hodge decomposition, Serre duality and Kodaira vanishing theorem.

\section{Index of transversally elliptic operator}

\subsection{Notations}

Let $G$ be a compact Lie group. We denote by $\hat{G}$ the set of equivalence classes of finite dimensional irreducible (complex) representations of $G$. For $\tau \in \hat{G}$, we denote by $V_{\tau}$ the representation space of $\tau$ and by $\chi_{\tau}$ its character: $\chi_{\tau}(g) = \text{Tr}(g \colon V_{\tau} \to V_{\tau})$. More generally, we denote the character of a representation space $V$ of $G$ by $\chi^G(V)$, or simply $\chi(V)$. Let the dual representation of $\tau$ in $V_{\tau}^*=\homo(V_{\tau}, \, \mathbb{C})$ be denoted by $\tau^*$. We denote by $R(G)$ the representation ring of $G$.

Let $\mathcal{D}(G)$ be the space of $\mathbb{C}$-valued smooth functions on $G$ and $\mathcal{D}'(G)$ be the space of distributions on $G$. $\mathcal{D}(G)$ is naturally a subspace of $\mathcal{D}'(G)$.

\subsection{Indices of transversally elliptic operators}

We recall the notion of transversally elliptic operator. See also \cite{Atiyahtrans} , \cite{Vergne08}.

Let $M$ be a compact smooth manifold with a smooth action of a compact Lie group $G$. 
$\mathfrak{g}$ denotes the Lie algebra of $G$.
For $X \in \mathfrak{g}$, we denote by $VX$ the vector field on $M$ defined by 
\[
  (VX)_x(f) = \left. \frac{d}{dt} f \left( \exp(-tX) \cdot x \right) \right|_{t = 0} \quad \quad (x \in M)
\]
for any smooth function $f$ defined around $x$.
We consider the closed subset $T_G^*M$ of the cotangent bundle $T^*M$, given by the union of the spaces 
$(T_G^*M)_x \subset T_x^*M$ $(x \in M)$ consisting of cotangent vectors orthogonal to the tangent space at $x$ to the orbit $G \cdot x$,\, i.e.,
\[
  T_G^*M = \left\{ \omega \in T^*M \left| \, \omega (VX) = 0 ,\, \text{for all} \; X \in \mathfrak{g} \right. \right\} .
\]
Note that $T_G^*M$ is not a subbundle of $T^*M$ in general since the dimension of the orbit may depend on the base point $x \in M$.

Let $\mathcal{E}^{\pm}$ be two $G$-equivariant complex vector bundles over $M$, that is, $G$ acts on total spaces $\mathcal{E}^{\pm}$ and $M$, and the actions of $G$ on $\mathcal{E}^{\pm}$ cover the action of $G$ on $M$ and are linear on fibers. We denote by $\Gamma (M,\, \mathcal{E}^{\pm})$ the spaces of smooth sections of $\mathcal{E}^{\pm}$.

Let $A \colon \Gamma (M,\, \mathcal{E}^{+}) \to \Gamma (M,\, \mathcal{E}^{-})$ be a $G$-invariant pseudo-differential operator of order $m$. Let $\pi \colon T^*M \to M$ be the natural projection. The principal symbol $\sigma (A)$ is a bundle map ${\pi}^*\mathcal{E}^+ \to {\pi}^*\mathcal{E}^-$ which is homogeneous of degree $m$, defined over $T^*M \setminus M$.
The operator $A$ is said to be transversally elliptic relative to $G$, or simply $G$-transversally elliptic, if its principal symbol $\sigma (A) (x,\, \xi )$ is invertible for all $(x,\, \xi ) \in T_G^*M$ such that $\xi \neq 0$. In paticular, we say that $A$ is ellptic if its principal symbol $\sigma (A) (x,\, \xi )$ is invertible for all $(x,\, \xi ) \in T^*M$ such that $\xi \neq 0$.
When $A$ is elliptic (or transversally elliptic), if we choose $G$-invariant metrics on $M$ and $\mathcal{E}^{\pm}$, the formal adjoint operator $A^*$ is also elliptic (or transversally elliptic).

It is well-known that if $A$ is a $G$-invariant elliptic operator, its kernel $\knl A$ is finite dimensional. In addition, it is a representation space of $G$ since $A$ is $G$-invariant. Thus in this case, the equivariant index of $A$ is defined by 
\[
  \ind_G A := \left[ \knl A \right] - \left[ \knl A^* \right] \quad \in R(G).
\]
Taking the trace of the above index, i.e., its character,  $\ind_G A$ can be regarded as a smooth function on $G$.

When $A$ is a $G$-transversally elliptic operator, in general, $\knl A$ is not necessarily finite dimensional. However, it has a generalized character. Indeed, for any irreducible representation $(\tau ,\, V_{\tau})$ of $G$, the multiplicity $n_{\tau}(A) := \text{dim} \left( \homo_G(V_{\tau},\, \knl A) \right)$ is finite, and the series
\[
  \sum_{\tau \in \hat{G}} n_{\tau}(A) \, \chi_{\tau}
\]
converges weakly in the space of distributions on $G$ (see \cite{Atiyahtrans}, \cite{Shubin}).

\begin{dfn} \label{transindex}
For a transversally elliptic operator $A$ relative to $G$, we define the generalized character of $\knl A$ as a distribution on $G$ by setting
\[
  \mathrm{character}(\knl A) := \sum_{\tau \in \hat{G}} n_{\tau}(A) \, \chi_{\tau} .
\]
And we also define the transversal index of $A$ as follows:
\[
  \ind_G A := \mathrm{character}(\knl A) - \mathrm{character}(\knl A^*) \quad \in \mathcal{D}'(G) .
\]
\end{dfn}

As in \cite{Atiyahtrans} and \cite{A-S1}, when $A$ is $G$-invariant elliptic, its principal symbol $\sigma(A)$ defines a class $[\sigma(A)] \in K_G(T^*M)$, and the index $\ind_GA$ depends only on this class $[\sigma(A)]$. Moreover each element in $K_G(T^*M)$ is represented by a class $[\sigma(A)]$ of some $G$-invariant elliptic pseudo-differential operator $A$. Thus we can define the index map:
\[
  \ind_G \colon K_G(T^*M) \to \mathcal{D}(G)
\]
by setting $\ind_G([\sigma(A)]) := \ind_GA$. 
Similarly, when $A$ is $G$-transversally elliptic, its principal symbol $\sigma(A)$ defines a class $[\sigma(A)|_{T_G^*M}] \in K_G(T_G^*M)$, and the index $\ind_GA$ depends only on this class $[\sigma(A)|_{T_G^*M}]$. In addition, each element in $K_G(T_G^*M)$ is represented by a class $[\sigma(A)|_{T_G^*M}]$ of some $G$-transversally elliptic pseudo-differential operator $A$. Thus we can define the index map:
\[
  \ind_G \colon K_G(T_G^*M) \to \mathcal{D}'(G)
\]
by setting $\ind_G([\sigma(A)|_{T_G^*M}]) := \ind_GA$. 

\begin{rem} \label{module}
$K_G(T_G^*M)$ is regarded as a module over the representation ring $R(G)$ by using the tensor products.
We define an action of $R(G)$ on $\mathcal{D}'(G)$ as multiplication by its character.
Then it turns out that $\ind_G \colon K_G(T_G^*M) \to \mathcal{D}'(G)$ is an $R(G)$-module homomorphism.

\end{rem}

\subsection{Free action property} \label{freeaction}
Transversal index has several important properties. Let us introduce one of them which we will use later. In \cite{Vergne08}, it is called ``Free action property". Since we have explained it roughly in the introduction, we explain that property by using $K$-groups in this section.

Let $G$ and $H$ be two compact Lie groups, and $P$ be a compact manifold on which $G \times H$ acts smoothly. We assume that the action of $H$ is free. Then $M := P/H$ is a compact manifold provided with an action of $G$ and the quotient map $p \colon P \to M$ is a principal $H$-bundle with an opposite action of $H$ (see Remark \ref{action} at the end of this section) and $G$-equivariant.
In addition, the pull-back of $p$ induces the canonical isomorphism $T_{G}^*M \cong \left( T_{G \times H}^*P \right) /H$ as $G$-equivariant vector bundles over $M$. Thus there is an isomorphism
\[
  p^* \colon K_G(T_G^*M) \stackrel{\cong}{\to} K_{G \times H}(T_{G \times H}^*P).
\]
Note that this isomorphism gives a correspondence between $G$-transversally elliptic operators on $M = P/H$ and $G \times H$-transversally elliptic operators on $P$.
Let $\mathcal{E}^{\pm}$ be two $G$-equivariant complex vector bundles over $M$, and $\sigma \colon \pi^* \mathcal{E}^+ \to \pi^*\mathcal{E}^-$ be a $G$-transversally elliptic symbol. For any finite dimensional irreducible representation $(\tau ,\, V_{\tau})$ of $H$, we construct a $G$-equivariant complex vector bundle $\underline{V_{\tau}} := P \underset{H}{\times} V_{\tau}$ $=(P \times V_{\tau}) / (q, \, v) \sim \left(q \cdot h ,\, \tau(h^{-1})(v) \right)$ over $M$. We consider the morphism
\[
  \sigma \otimes \mathrm{Id}_{V_{\tau}} \colon 
    \pi^*(\mathcal{E}^+ \otimes \underline{V_{\tau}}) \to \pi^*(\mathcal{E}^- \otimes \underline{V_{\tau}}) ,
\]
which is $G$-transversally elliptic. We denote by $[\sigma] \otimes V_{\tau}$ the equivalence class $[\sigma \otimes \mathrm{Id}_{V_{\tau}}]  \in K_G(T_G^*M)$.

\begin{thm}[\cite{Atiyahtrans}]
For any $[\sigma] \in K_G(T_G^*M)$, we have
\[
  \ind_{G \times H} (p^*[\sigma]) = \sum_{\tau \in \hat{H}} \ind_G ( [\sigma] \otimes V_{\tau^*} ) \, \chi_{\tau} 
    \quad \in \mathcal{D}'(G \times H).
\] 
\end{thm}

As a special case, we consider the isomorphism
\[
  p^* \colon K_G(T^*M) \stackrel{\cong}{\to} K_{G \times H}(T_H^*P).
\]
This implies that operators on $P$ which are $G \times H$-invariant $H$-transversally elliptic can be reduced to $G$-invariant elliptic operators on $M$.

\begin{cor}[\cite{Atiyahtrans}] \label{spfap}
For any $[\sigma] \in K_G(T^*M)$, we have
\[
  \ind_{G \times H} (p^*[\sigma]) = \sum_{\tau \in \hat{H}} \ind_G ([\sigma] \otimes V_{\tau^*}) \, \chi_{\tau} 
    \quad \in \mathcal{D}'(G \times H).
\]
\end{cor}

Our explicit formula (Theorem \ref{main}) will be proved by using this corollary in the next section.

\begin{rem} \label{spform}
As Atiyah explained in \cite{Atiyahtrans}, index of any element in $K_{G \times H}(T_{H}^*M)$ can be regarded as a linear map as follows:
If $A$ is a transversally elliptic operator relative to $H$ and $G$-invariant, then it is of course transversally elliptic relative to $G \times H$ so $\ind_{G \times H} A \in \mathcal{D}'(G \times H)$.
In this case, however it is a distribution of special form as below.
If we write
\[
  \ind_{G \times H} A = \sum_{\alpha ,\, \beta} C_{\alpha ,\, \beta} \chi_{\alpha} \chi_{\beta} ,
\]
where $C_{\alpha ,\, \beta} \in \mathbb{Z}$, $\alpha \in \hat{G}$, $\beta \in \hat{H}$ (Note that by definition, $\ind_{G \times H} A$ can be written in this form), then the partial sums for a fixed $\beta$ are finite, that is, $C_{\alpha ,\, \beta} = 0$ for all but finitely many $\alpha$ because $A$ is originally transversally elliptic relative to $H$.  This means that $\ind_{G \times H} A$ can be viewed as a continuous linear map $\mathcal{D}(H) \to \mathcal{D}(G)$ given by 
\[
  \varphi \mapsto \sum C_{\alpha ,\, \beta} \chi_{\alpha} \left( \int_{H} \varphi \chi_{\beta} \right) .
\]
If we denote this map by $\ind_{G,\, H} A \colon \mathcal{D}(H) \to \mathcal{D}(G)$, then 
\[
  \left( \ind_{G \times H} A \right) (\psi \varphi) = \int_{G} \psi \,  \left( \ind_{G,\, H} A \right) (\varphi) 
    \quad \in \mathbb{C} ,
\]
for $\psi \in \mathcal{D}(G)$ and $\varphi \in \mathcal{D}(H)$.
Note that $\ind_{G,\, H} A$ is determined by its restriction to $R(H)$ and that $\overline{\chi_{\beta}}$ are mapped to $\sum C_{\alpha ,\, \beta} \chi_{\alpha}$ by $\ind_{G,\, H} A$.

Our main theorem can be also viewed as an formula for linear maps $\mathcal{D}(S^1) \to \mathcal{D}(T^{n+1})$.
\end{rem}

\begin{rem} \label{action}
We always assume that actions of Lie groups on manifolds are left actions, and actions on principal bundles are right actions. Under these conventions, when a manifold $P$ has a free (left) action of a compact Lie group $G$, $\pi \colon P \to P/G$ is a principal $G$-bundle with the following right action of $G$:
\[
  p \cdot g := g^{-1} \cdot p \quad (p \in P, \, g \in G).
\]
In other words, taking a local trivialization of P:
\[
  \phi \colon \pi^{-1}(U) \stackrel{\cong}{\to} U \times G ,
\]
we have $\phi(g^{-1} \cdot p) = \phi(p \cdot g) = (x,\, ag)$ ,where $\phi(p) = (x,\, a)$.

Conversely, when a principal $G$-bundle $P \to M$ is given, and we regard the total space P as a $G$-manifold, the left action of $G$ is given by
\[
  g \cdot p := p \cdot g^{-1} \quad (p \in P ,\, g \in G).
\]
\end{rem}

\section{The index formula}

\subsection{Settings} \label{setting}
In this paper, to simplify the notations, we may identify an element of the representation ring $R(G)$ of a compact Lie group $G$ with its virtual character via the following map: 
\[
  R(G) \ni [V^0] - [V^1] \mapsto \chi(V^0) - \chi(V^1) \in \mathcal{D}(G).
\]
Let $R(S^1)$ and $R(T^{n+1})$ be the representation rings of $S^1$ and $T^{n+1}$ respectively. 
Then the following lemma is well-known.
\begin{lem}
There are ring isomorphisms
\[
  R(S^1) \cong \mathbb{Z} [t,\, t^{-1}] \quad , \quad 
    R(T^{n+1}) \cong \mathbb{Z} \left[ t_1 ,\, t_1^{-1} ,\, \dots ,\, t_{n+1} ,\, t_{n+1}^{-1} \right] .
\]
\begin{enumerate}
  \item $t$ corresponds to the character of the irreducible unitary representation 
    $(\rho_1,\, \mathbb{C})$ of $S^1$ defined by
      $
        \rho_1(u) := u \in U(1) \quad (u \in S^1).
      $

  \item $t_j$ corresponds to the character of the irreducible unitary representation 
    $(\rho_j,\, \mathbb{C})$ of $T^{n+1}$ defined by
      $
        \rho_j(u) := u_j \in U(1) \quad (u = (u_1, \, \dots ,\, u_{n+1}) \in T^{n+1}).
      $
\end{enumerate}
\end{lem}

%

%


Next, recall the tautological line bundle over $\mathbb{CP}^n$:
\[
  \mathcal{O}(-1) := \left\{ (\bm{l} \, ,\, v ) \in \mathbb{CP}^n \times \mathbb{C}^{n+1} \left| \, v \in \bm{l} \right. \right\}
\]
as a subbundle of $\mathbb{CP}^n \times \mathbb{C}^{n+1}$,
where we regard $\bm{l} \in \mathbb{CP}^n$ as a $1$-dimensional linear subspace of $\mathbb{C}^{n+1}$.
Moreover, recall that $\mathcal{O}(1)$ is the dual $\mathcal{O}(-1)^*$.
For $k>0$, let $\mathcal{O}(k)$ be the line bundle $\mathcal{O}(1) \otimes \dots \otimes \mathcal{O}(1)$\, ($k$-times).
Analogously, for $k<0$, we define $\mathcal{O}(k) = \mathcal{O}(-k)^* := \mathcal{O}(-1) \otimes \dots \otimes \mathcal{O}(-1)$\, ($|k|$-times), and for $k = 0$, 
we define $\mathcal{O}(0)$ as the trivial line bundle $\mathbb{CP}^n \times \mathbb{C}$.

\begin{rem} \label{hermitian}
Since $\mathcal{O}(-1)$ has the canonical Hermitian metric induced by the one of $\mathbb{CP}^n \times \mathbb{C}^{n+1}$, also its dual $\mathcal{O}(1)$ has the canonical Hermitian metric.
Hence, all of $\mathcal{O}(k)$ have the canonical Hermitian metrics.
\end{rem}

Let $S\left( \mathcal{O}(k) \right)$ denote the circle bundle of $\mathcal{O}(k)$, that is,
\[
  S\left( \mathcal{O}(k) \right) := \left\{ v \in \mathcal{O}(k) \left| \, \Vert v \Vert = 1 \right. \, \right\},
\]
where $\Vert \cdot \Vert$ denotes the norm induced by the Hermitian metric.
Then, for each $k \in \mathbb{Z}$, we have $S^1$-principal bundle:
\[
  p_k \colon S\left( \mathcal{O}(k) \right) \to \mathbb{CP}^n
\]
equipped with the right action of $S^1$ on each fiber by scalar multiplication.
Note that when $S(\mathcal{O}(k))$ is regarded as a $S^1$-manifold, the (left) action of $S^1$ is given by multiplying by inverse: $S^1 \ni a \colon v \mapsto a^{-1}v \, $ (see Remark \ref{action}).

The canonical action of $T^{n+1}$ on $\mathbb{C}^{n+1} \setminus \{ 0 \}$ induces an action of $T^{n+1}$ on $\mathbb{CP}^n$. The image of the point $\bm{l} = [z_1 : \dots : z_{n+1}] \in \mathbb{CP}^n$ under this action of $u = (u_1 ,\, \dots ,\, u_{n+1}) \in T^{n+1}$ is denoted by $u \cdot \bm{l} = [ u_1z_1 : \dots : u_{n+1}z_{n+1} ]$.
Moreover, $T^{n+1}$ acts on the principal bundles $p_k \colon S\left( \mathcal{O}(k) \right) \to \mathbb{CP}^n$. Indeed, for $k < 0$, $u =(u_1,\, \dots ,\, u_{n+1}) \in T^{n+1}$ acts on $\mathcal{O}(k)$ as follows:
\[
  \begin{array}{rcl}
  \mathcal{O}(k)|_{\bm{l}} = \bm{l}^{\otimes |k|} &\to 
    &\mathcal{O}(k)|_{u \cdot \bm{l}} = \left( u \cdot \bm{l} \right)^{\otimes |k|} \\ [5pt]
  v^{(1)} \otimes \dots \otimes v^{(-k)} &\mapsto &(u \cdot v^{(1)}) \otimes \dots \otimes (u \cdot v^{(-k)}) ,
  \end{array}
\]
where $v^{(i)} = (v_1^{(i)},\, \dots ,\, v_{n+1}^{(i)}) \in \bm{l} \subset \mathbb{C}^{n+1}$, and $u \cdot v^{(i)} := (u_1v_1^{(i)},\, \dots ,\, u_{n+1}v_{n+1}^{(i)})$.
Since these actions are isometric (with respect to the Hermitian metrics which are explained in Remark \ref{hermitian}), we can restrict them to $S(\mathcal{O}(k))$.
When $k > 0$, we consider the dual action of the above one. Finally for $k = 0$, the action of $T^{n+1}$ on $\mathbb{CP}^n$ induces the one on $\mathcal{O}(0) = \mathbb{CP}^n \times \mathbb{C}$ which are trivial on each fiber.
Thus $p_k \colon S\left( \mathcal{O}(k) \right) \to \mathbb{CP}^n$ are $T^{n+1}$-equivariant principal $S^1$-bundles.

\subsection{Dolbeault operator} \label{dolop}

We consider the Dolbeault complex.
Let $\displaystyle \mathcal{E}^+ := \bigoplus_{q : even} \bigwedge\nolimits^q \overline{T_{\mathbb{C}}^*\mathbb{CP}^n}, \quad \mathcal{E}^- := \bigoplus_{q : odd} \bigwedge\nolimits^q \overline{T_{\mathbb{C}}^*\mathbb{CP}^n}$, where $T_{\mathbb{C}}^*\mathbb{CP}^n$ is the holomorphic cotangent bundle of $\mathbb{CP}^n$. 
Since $T^{n+1}$ acts on $\mathbb{CP}^n$ as explained above, $\mathcal{E}^{\pm}$ are $T^{n+1}$-equivariant complex vector bundles over $\mathbb{CP}^n$. Let us define the Dirac operator $D_n$ on $\mathbb{CP}^n$ associated with the Dolbeault operator $\bar{\, \partial}$ on $\mathbb{CP}^n$ by
\[
  D_n := \bar{\, \partial} + \bar{\, \partial}^* \colon \Gamma(\mathbb{CP}^n,\, \mathcal{E}^+)
    \to \Gamma(\mathbb{CP}^n,\, \mathcal{E}^-) .
\]
In this paper, we often call $D_n$ simply the Dolbeault operator.
Since $\bar{\, \partial}$ on $\mathbb{CP}^n$ is elliptic and $T^{n+1}$-invariant, so is $D_n$. 
Then from Corollary \ref{spfap}, the transversal index of $p_k^*[D_n] \in K_{T^{n+1} \times S^1}(T_{S^1}^* \mathbb{CP}^n)$ is given by
\begin{equation} \label{key1}
  \ind_{T^{n+1} \times S^1} \left( p_k^*[D_n] \right) = 
    \sum_{m \in \mathbb{Z}} \ind_{T^{n+1}} \left( [ D_n ] \otimes \underline{t}^{-m} \right) \, t^m .
\end{equation}
Here $\underline{t}^{-m}$ is a complex line bundle over $\mathbb{CP}^n$ defined by $S(\mathcal{O}(k)) \underset{S^1}{\times} t^{-m} = \left( S(\mathcal{O}(k)) \times t^{-m} \right) / \sim$, where the equivalence relation is given by 
\[
  \left( v ,\, \xi \right) \sim \left( av ,\, \left( a^{-1} \right)^{-m} \xi \right)  \quad \text{for some} \; a \in S^1 .
\] 
In addition, this line bundle has the action of $T^{n+1}$ induced by the action on $S(\mathcal{O}(k))$.

\begin{lem}  \label{holline}
There is a $T^{n+1}$-equivariant isomorphism of complex line bundles over $\mathbb{CP}^n$:
\[
  \underline{t}^{-m} = S(\mathcal{O}(k)) \underset{S^1}{\times} t^{-m} \cong \mathcal{O}(-km) \qquad (m \in \mathbb{Z})
\]
In particular, $\underline{t}^{-m}$ admits a structure of holomorphic line bundle.
\end{lem}

\begin{proof}
This lemma follows from the more general statement below.

\end{proof}

Let $L$ be a complex line bundle over a complex manifold $M$ with a Hermitian metric, on which a compact Lie group $G$
acts isometrically.
We denote by $P_L$ the principal $S^1$-bundle associated with $L$, that is,
\[
  P_L := \bigcup_{x \in M} P_{L,\, x} 
   ,
\]
where 
$P_{L,\, x}$ is the space consisting of unitary isomorphisms $\phi$ from $\mathbb{C}$ to the fiber $L|_x$ of $L$ on $x \in M$.
The right action of $S^1$ on $P_L$ is given by composing from right:
\[
  u \colon P_{L,\, x} \ni \phi \mapsto \phi \circ u \in P_{L,\, x} \quad (u \in S^1) ,
\]
where each $u \in S^1$ is regarded as a linear map $\mathbb{C} \to \mathbb{C}$ by scalar multiplication.
In addition, let $S(L)$ be the circle bundle of $L$ which is a principal $S^1$-bundle with respect to the right action of $S^1$ given by scalar multiplication.
Since $G$ acts on $L$ isometrically, both $P_L$ and $S(L)$ are $G$-equivariant principal $S^1$-bundles. Note that the actions of $G$ on them are left actions.
Then, the following lemma holds.
\begin{lem}
\begin{enumerate}
  \item There is a $G$-equivariant isomorphism of principal $S^1$-bundles:
    \[
      S(L) \cong P_L .
    \]
  \item For any $l \in \mathbb{Z}$, there exists a $G$-equivariant isomorphism of complex vector bundles:
    \[
      P_L \underset{S^1}{\times} t^l \cong L^{\otimes l} .
    \]
\end{enumerate}
\end{lem}

\begin{proof}
\begin{enumerate}
  \item Let us define a map $\Phi \colon S(L) \to P_L$ as follows:
For $u_x \in S(L|_x) = S(L)|_x \subset S(L)$, define $\Phi(u_x) \in P_{L ,\, x}$ by setting $\Phi(u_x)(\xi) := u_x \xi \in L|_x \, (\xi \in \mathbb{C})$.
Since we can verify that $\Phi$ is $S^1$-equivariant and $G$-equivariant, this $\Phi$ gives a required isomorphism.
  \item 
  It is obvious that there exists an isomorphism if we forget the $G$-equivariance. In addition, because $G$ acts on $L$ isometrically, under an isometric local trivialization, the action of $g \in G$ is represented by the scalar multiplication of an element in $S^1$. Thus we find that this isomorphism is $G$-equivariant.
\end{enumerate}
\end{proof}

\begin{lem} \label{eqindex}
For $n \geq 1$, $k,\, m \, \in \mathbb{Z}$, we have
\begin{equation} \label{eqindexoncpn}
  \ind_{T^{n+1}} \left( [D_n] \otimes \underline{t}^{-m} \right) = 
    \sum_{q = 0}^n (-1)^q H^q(\mathbb{CP}^n,\, \mathcal{O}(-km)) \quad \in R(T^{n+1}) .
\end{equation}
In particular, 
\begin{equation} \label{key2}
  \ind_{T^{n+1} \times S^1} \left( p_k^*[ D_n ] \right) = 
    \sum_{m \in \mathbb{Z}} \left( \sum_{q = 0}^n (-1)^q H^q(\mathbb{CP}^n,\, \mathcal{O}(-km)) \right) \, t^m .
\end{equation}
\end{lem}

\begin{proof}
First, (\ref{key2}) immediately follows from (\ref{key1}) and (\ref{eqindexoncpn}). We will verify the equation (\ref{eqindexoncpn}). From Lemma \ref{holline}, we find that $[D_n] \otimes \underline{t}^{-m} \in K_{T^{n+1}}(T^*\mathbb{CP}^n)$ is represented by the Dirac operator with the coefficient $\mathcal{O}(-km)$:
\[
  D_{n,\, \mathcal{O}(-km)} := \bar{\partial}_{\mathcal{O}(-km)} + \bar{\partial}_{\mathcal{O}(-km)}^* 
    \colon \Gamma(\mathbb{CP}^n,\, \mathcal{E}^+ \otimes \mathcal{O}(-km)) 
      \to \Gamma(\mathbb{CP}^n,\, \mathcal{E}^- \otimes \mathcal{O}(-km)) ,
\]
which is $T^{n+1}$-invariant and elliptic, where $\bar{\partial}_{\mathcal{O}(-km)} := \bar{\partial} \otimes \mathrm{Id}_{\mathcal{O}(-km)}$ and $\bar{\partial}_{\mathcal{O}(-km)}^* := \bar{\partial}^* \otimes \mathrm{Id}_{\mathcal{O}(-km)}$.
In addition, using Hodge decomposition (see \cite{cpxgeo} for details), we have $T^{n+1}$-equivariant isomorphisms:
\begin{align*}
  \knl \left( D_{n,\, \mathcal{O}(-km)} \colon \Gamma(\mathbb{CP}^n,\, \mathcal{E}^+ \otimes \mathcal{O}(-km)) 
    \to \Gamma(\mathbb{CP}^n,\, \mathcal{E}^- \otimes \mathcal{O}(-km)) \right) 
      &\cong \bigoplus_{q : even} H^q(\mathbb{CP}^n,\, \mathcal{O}(-km)) , \\
  \knl \left( D_{n,\, \mathcal{O}(-km)}^* \colon \Gamma(\mathbb{CP}^n,\, \mathcal{E}^- \otimes \mathcal{O}(-km)) 
    \to \Gamma(\mathbb{CP}^n,\, \mathcal{E}^+ \otimes \mathcal{O}(-km)) \right) 
      &\cong \bigoplus_{q : odd} H^q(\mathbb{CP}^n,\, \mathcal{O}(-km)) .
\end{align*}
These imply (\ref{eqindexoncpn}).
\end{proof}

Let us introduce the following notation.

\begin{dfn} \label{chi}
For $n\geq 1$,\, $l \in \mathbb{Z}$, we define 
\begin{equation} \label{defchi}
  \chi_{n,\,l} := \sum_{q=0}^n (-1)^q H^q(\mathbb{CP}^n,\, \mathcal{O}(l)) \quad \in R(T^{n+1}) .
\end{equation}
\end{dfn}

Then, from (\ref{key2}), we can express the index of $p_k^*[ D_n ] \in K_{T^{n+1} \times S^1}(T_{S^1}^*\mathbb{CP}^n)$ as follows:
\begin{equation} \label{key3}
  \ind_{T^{n+1} \times S^1} \left( p_k^*[D_n] \right) = \sum_{m \in \mathbb{Z}} \chi_{n,\, -km} \, t^m .
\end{equation}

\subsection{Calculation of cohomology groups}

In order to compute the transversal index of $p_k^*[ D_n ]$, we have to obtain the character of the representation space $H^q(\mathbb{CP}^n,\, \mathcal{O}(m))$ of $T^{n+1}$. See also \cite{cpxgeo}.

%
%
%
%
%
%
%
%
%
%

\begin{lem} \label{K}
Let $K_n$ denote the canonical line bundle of $\mathbb{CP}^n$. Then there is an isomorphism of holomorphic vector bundles over $\mathbb{CP}^n$:
\[
K_n \cong \mathcal{O}(-n-1).
\]
\end{lem}

\begin{proof}
It is sufficient to check that $\mathcal{O}(n+1)$ is isomorphic to the determinant bundle $\bigwedge\nolimits^n T\mathbb{CP}^n$. Thus computation of both transition functions gives the conclusion (see \cite{cpxgeo} for details).
\end{proof}

\begin{rem} \label{notequiv}
There is no $T^{n+1}$-equivariant isomorphism between $K_n$ and $\mathcal{O}(-n-1)$.
Indeed, since $p_1 := [1 : 0 : \dots : 0] \in \mathbb{CP}^n$ is a fixed point of the action of $T^{n+1}$, the fibers $K_n|_{p_1}$ and $\mathcal{O}(-n-1)|_{p_1}$ on $p_1$ are representation spaces of $T^{n+1}$. It is easy to check that the characters of these representations are given by $\displaystyle \frac{t_1^n}{t_2 \dots t_{n + 1}}$ and $t_1^{n + 1}$ respectively.
\end{rem}

\begin{prop} \label{cohomology}
For $n \geq 1$, and $m \in \mathbb{Z}$, we have the following isomorphisms of representations of $T^{n+1}$:
\begin{itemize}
  \item $q=0$ 
    \[H^0(\mathbb{CP}^n, \, \mathcal{O}(m)) \cong
      \left\{
        \begin{split}
          &S^m((\mathbb{C}^{n+1})^*) &\quad &(m \geq 0) \, , \\
          &\qquad 0 &\quad &(m < 0) \, ,
        \end{split}
      \right. 
    \]

  \item $0<q<n$
    \[
      H^q(\mathbb{CP}^n, \, \mathcal{O}(m)) \cong 0 \, ,
    \]

  \item $q=n$
    \[
      H^n(\mathbb{CP}^n, \, \mathcal{O}(m)) \cong
        \left\{ \begin{split}
          &\qquad 0 &\quad &(m>-n-1) \, , \\
          &(\bigwedge\nolimits^{n+1} \mathbb{C}^{n+1}) \otimes S^{-m-n-1}(\mathbb{C}^{n+1}) &\quad &(m \leq -n-1) \, ,
        \end{split} \right.
    \]
\end{itemize}
where $S^i(V)$ is the symmetric algebra of a vector space $V$ of degree $i$ and we regard $\mathbb{C}^{n+1}$ as a representation space of $T^{n+1}$ defined by a canonical action of $T^{n+1}$:
\[
  T^{n+1} \ni (u_1,\, \dots , \, u_{n+1}) \colon (z_1,\, \dots ,\, z_{n+1}) \mapsto (u_1z_1,\, \dots ,\, u_{n+1}z_{n+1}).
\]
\end{prop}

The next corollary immediately follows from this proposition.

\begin{cor}\label{chicharacter}
Let $n \geq 1$, and $l \in \mathbb{Z}$. Then,
\[
  \chi_{n,\,l} =
    \left\{
      \begin{split} 
        &\quad \sum_{\substack{r_1 + \dots + r_{n+1} = l \\ r_1,\, \dots,\, r_{n+1} \geq 0}} 
          t_1^{-r_1} \dots t_{n+1}^{-r_{n+1}} &\quad &(l \geq 0) \quad , \\[5pt]
        &\qquad \qquad 0 &\quad &(-n-1 < l < 0) \quad , \\[5pt]
        &(-1)^n t_1 \dots t_{n+1} \sum_{\substack{r_1 + \dots + r_{n+1} = -l-n-1 \\ r_1,\, \dots,\, r_{n+1} \geq 0}} 
          t_1^{r_1} \dots t_{n+1}^{r_{n+1}} &\quad &(l \leq -n-1) \quad .
      \end{split}
    \right.
\]
\end{cor}

\begin{rem}
If we set $n=0$ in the right hand side of the above equation, we obtain $t_1^{-l}$ \, for all $(l \in \mathbb{Z})$.
This is why, we define  
\[
  \chi_{0,\, l} := t_1^{-l} \quad (\, l \in \mathbb{Z} \, )
\]
for convenience.
\end{rem}

\begin{proof}[Proof of Proposition \ref{cohomology}]
Let us recall Kodaira vanishing theorem: for any positive line bundle $L$ over an $n$-dimensional compact K\"{a}hler manifold $M$, we obtain $H^{q}(M,\, \Omega_M^p \otimes L) \cong 0$ for $p + q > n$.
Here we denote by $\Omega_M^p$ the space of holomorphic $p$-forms on $M$ for $p = 0 ,\, \dots ,\, n$.
If $0 < q \leq n$ and $m > -n-1$, then from Lemma \ref{K} and Kodaira vanishing theorem, we have
\[
  H^q(\mathbb{CP}^n,\, \mathcal{O}(m)) \cong H^q(\mathbb{CP}^n,\, K_n \otimes \mathcal{O}(m+n+1)) \cong 0.
\]
Moreover if $0 \leq q <n$ and $m < 0$, using Kodaira vanishing theorem, we have 
\[
  H^q(\mathbb{CP}^n,\, \mathcal{O}(m)) \cong H^{n-q}(\mathbb{CP}^n,\, K_n \otimes \mathcal{O}(-m))^* \cong 0.
\]
Here the first isomorphism follows from Serre duality, which claims that for any holomorphic vector bundle $E$ over an $n$-dimensional compact complex manifold $M$, there exist natural isomorphisms $H^q(M,\, \Omega_M^p \otimes E) \cong H^{n-q}(M,\, \Omega_M^{n-p} \otimes E^*)^*$.

The rest cases follow from the below.
\end{proof}

\begin{lem} \label{cohomology0}
For any integer $m \geq 0$, we have an isomorphism of representations of $T^{n+1}$:
\[
  H^0(\mathbb{CP}^n,\, \mathcal{O}(m)) \cong S^m((\mathbb{C}^{n+1})^*) \, ,
\]
where $\mathbb{C}^{n+1}$ is provided with the canonical action of $T^{n+1}$ (as in Proposition \ref{cohomology}).
\end{lem}

\begin{proof}
Let $(z_1,\, \dots ,\, z_{n+1})$ be the standard global coordinate of $\mathbb{C}^{n+1}$, and $\mathbb{C}[z_1,\, \dotsm,\, z_{n+1}]_m$ be the space of homogeneous polynomials of degree $m$.
We can regard $\alpha \in \mathbb{C}[z_1,\, \dotsm,\, z_{n+1}]_m$ as a linear function on $(\mathbb{C}^{n+1})^{\otimes m}$.
Indeed, for $1 \leq j_1 \leq j_2 \leq \dots \leq j_m \leq n+1$, a homogeneous polynomial $z_{j_1}z_{j_2}\dots z_{j_m}$ defines the following $\mathbb{C}$-valued linear function on $(\mathbb{C}^{n+1})^{\otimes m}$:


\begin{align*}
  z_{j_1}z_{j_2}\dots z_{j_m} \colon v^{(1)} \otimes v^{(2)} \otimes \dots \otimes v^{(m)} \longmapsto 
    &\frac{1}{m!} \sum_{\sigma \in S_m} z_{j_1}(v^{(\sigma(1))}) z_{j_2}(v^{(\sigma(2))}) \dots z_{j_m}(v^{(\sigma(m))}) \\
    &= \frac{1}{m!} \sum_{\sigma \in S_m} v_{j_1}^{(\sigma(1))} v_{j_2}^{(\sigma(2))} \dots v_{j_m}^{(\sigma(m))} ,
\end{align*}
where $v^{(i)} = (v_1^{(i)},\, \dotsm ,\, v_{n+1}^{(i)})$ denotes an element of $\mathbb{C}^{n+1}$, and $S_m$ denotes the $m$-th symmetric group.
Thus each $\alpha \in \mathbb{C}[z_1,\, \dotsm,\, z_{n+1}]_m$ defines a holomorphic map $\mathbb{CP}^n \times (\mathbb{C}^{n+1})^{\otimes m} \to \mathbb{C}$ which is linear on each fiber $(\mathbb{C}^{n+1})^{\otimes m}$.
Since $\mathcal{O}(-m)$ is a subbundle of $\mathbb{CP}^n \times (\mathbb{C}^{n+1})^{\otimes m}$, restricting to $\mathcal{O}(-m)$, we get a holomorphic section of $\mathcal{O}(m)$.
In this way, we can define a linear map 
\[
  \varphi \colon \mathbb{C}[z_1,\, \dotsm ,\, z_{n+1}]_m \to 
    \Gamma_{\rm{hol}}(\mathbb{CP}^n,\, \mathcal{O}(m)) \cong H^0(\mathbb{CP}^n,\, \mathcal{O}(m)).
\]
Let us denote by $s_{j_1} \dotsm s_{j_m}$ the section corresponding to $z_{j_1} \dotsm z_{j_m}$, i.e. $\varphi (z_{j_1} \dotsm z_{j_m}) = s_{j_1} \dotsm s_{j_m}$.
Actually, this map $\varphi$ is bijective (see \cite{cpxgeo} Proposition 2.4.1 for details).  Finally we calculate the characters of these representation spaces.
In general, when $E \to M$ is a $G$-equivariant vector bundle, $G$ acts on the space of sections $\Gamma(M,\, E)$ as follows:
\[
  \left( g \cdot s \right) (x) := g \left( s (g^{-1} \cdot x) \right) \quad \in E_x
\]
for $g \in G$ and $s \in \Gamma(M,\, E)$. 
Let $u = (u_1,\, \dots ,\, u_{n+1}) \in T^{n+1}$. Then by the definition of $s_i$,
\[
  (u \cdot s_i)(\bm{l}) = u_i^{-1} (s_i(\bm{l})) \quad \in \mathcal{O}(1)|_{\bm{l}} = 
    \homo(\bm{l},\, \mathbb{C}) \quad \text{for} \quad \bm{l} \in \mathbb{CP}^n.
\] 
Analogously, 
\[
  \left( u \cdot (s_{j_1} \dotsm s_{j_m}) \right) (\bm{l}) = u_{j_1}^{-1} \dotsm u_{j_m}^{-1} (s_{j_1} \dotsm s_{j_m})(\bm{l}) 
    \quad \in \mathcal{O}(m)|_{\bm{l}} = \homo(\bm{l}^{\otimes m},\, \mathbb{C}).
\]
Thus we find that there exists an isomorphism $H^0(\mathbb{CP}^n,\, \mathcal{O}(m)) \cong S^m((\mathbb{C}^{n+1})^*)$ and their characters are given by $\displaystyle \sum_{\substack{ r_1 + \dots + r_{n+1} = m \\ r_1,\, \dots ,\, r_{n+1} \geq 0}} t_1^{-r_1} \dotsm t_{n+1}^{-r_{n+1}}$.
\end{proof}

\begin{lem} \label{cohomologyn}
For any integer $m \leq -n-1$, we have an isomorphism of representations of $T^{n+1}$:
\[
  H^n(\mathbb{CP}^n,\, \mathcal{O}(m)) \cong 
    \left( \bigwedge\nolimits^{n+1} \mathbb{C}^{n+1} \right) \otimes S^{-m-n-1}(\mathbb{C}^{n+1}) \, ,
\]
where $\mathbb{C}^{n+1}$ is given the same action of $T^{n+1}$ as in Lemma \ref{cohomology0}.
\end{lem}

\begin{proof}
Let $m \leq -n-1$, then from Serre duality, 
\begin{equation} \label{cohnkey1}
  H^n(\mathbb{CP}^n,\, \mathcal{O}(m)) \cong H^0(\mathbb{CP}^n,\, K_n \otimes \mathcal{O}(-m))^*.
\end{equation}
Since the pairing used in Serre duality is canonical, this isomorphism is $T^{n+1}$-equivariant.
In addition, using Lemma \ref{K}, we have an isomorphism of vector spaces
\begin{equation} \label{cohnkey2}
  H^0(\mathbb{CP}^n,\, K_n \otimes \mathcal{O}(-m)) \cong H^0(\mathbb{CP}^n,\, \mathcal{O}(-m-n-1)).
\end{equation}
Here this isomorphism is not $T^{n+1}$-equivariant because the isomorphism in Lemma \ref{K} is not $T^{n+1}$-equivariant (see Remark \ref{notequiv}). Thus we can not calculate the character of $H^n(\mathbb{CP}^n,\, \mathcal{O}(m))$ from the above isomorphism.
However since $-m-n-1 \geq 0$, we can obtain only the dimension of $H^n(\mathbb{CP}^n,\, \mathcal{O}(m))$ from Lemma \ref{cohomology0}, (\ref{cohnkey1}) and (\ref{cohnkey2}) as follows: 
\begin{align} \label{dimcoh}
  \dim H^n(\mathbb{CP}^n,\, \mathcal{O}(m)) = \dim H^0(\mathbb{CP}^n,\, \mathcal{O}(-m-n-1))^* = 
    \dim S^{-m-n-1}(\mathbb{C}^{n+1}) = \left( \substack{-m-1 \\ \\ n} \right) .
\end{align}
We will construct holomorphic sections of $K_n \otimes \mathcal{O}(-m)$ which form a basis of $H^0(\mathbb{CP}^n,\, K_n \otimes \mathcal{O}(-m))$.
Let $\displaystyle \mathbb{CP}^n = \bigcup_{i=1}^{n+1} U_i$, where $U_i = \{ [z_1 : \dotsm : z_{n+1}] \in \mathbb{CP}^n | z_i \neq 0 \}$ be the standard open covering of $\mathbb{CP}^n$ and $\varphi_i \colon U_i \to \mathbb{C}^n$ be the standard local coordinate given by
\[
  \varphi_i \left( [z_1 : \dotsm : z_{n+1}] \right) = 
    \left( \frac{z_1}{z_i},\ \dots ,\, \frac{z_{i-1}}{z_i},\, \frac{z_{i+1}}{z_i},\, \dots,\, \frac{z_{n+1}}{z_i} \right) .
\]
Let $(\zeta_1^{(i)},\, \dots ,\, \zeta_n^{(i)})$ be the coordinate of $\mathbb{C}^n = \varphi_i(U_i)$, then we have
\[
  \left( \zeta_1^{(i)},\, \dots ,\, \zeta_n^{(i)} \right) = 
    \left( \frac{z_1}{z_i},\ \dots ,\, \frac{z_{i-1}}{z_i},\, \frac{z_{i+1}}{z_i},\, \dots,\, \frac{z_{n+1}}{z_i} \right) .
\]
We will prove the following two claims at the end of this subsection.

\begin{clm} \label{transition}
On the intersection $U_i \cap U_j$, 
\[
  d\zeta_1^{(i)} \wedge \dots \wedge d\zeta_n^{(i)} = 
    (-1)^{i+j}\frac{z_j^{n+1}}{z_i^{n+1}} d\zeta_1^{(j)} \wedge \dots \wedge d\zeta_n^{(j)}.
\]
\end{clm}

Let us define local sections of $K_n \otimes \mathcal{O}(-m)$ as follows:
for $1 \leq i \leq n + 1$, and $(r_1,\, \dots ,\, r_{n + 1}) \in \mathbb{Z}_{\geq 0}^{n + 1}$ such that $r_1 + \dots + r_{n+1} = -m-n-1$, we define holomorphic sections on $U_i$ of $K_n \otimes \mathcal{O}(-m)$
\[
  \omega_{r_1,\, \dots ,\, r_{n + 1}}^i := 
    (-1)^{i + 1} d\zeta_1^{(i)} \wedge \dots \wedge d\zeta_n^{(i)} \otimes s_i^{n + 1} s_1^{r_1} \dots s_{n + 1}^{r_{n + 1}}.
\]
For each $(r_1,\, \dots ,\, r_{n + 1}) \in \mathbb{Z}_{\geq 0}^{n + 1}$ such that $r_1 + \dots + r_{n+1} = -m-n-1$, the family $\{ \omega_{r_1,\, \dots ,\, r_{n+1}}^i \}_{i = 1 ,\, \dots ,\, n+1}$ of local sections provides a global section $\omega_{r_1,\, \dots ,\, r_{n+1}}$ of $K_n \otimes \mathcal{O}(-m)$, that is,

\begin{clm} \label{gluing}
For $(r_1,\, \dots ,\, r_{n+1}) \in \mathbb{Z}_{\geq 0}^{n+1}$ such that $r_1 + \dots + r_{n+1} = -m-n-1$, the sections $\{ \omega_{r_1,\, \dots ,\, r_{n+1}}^i \}_{i = 1 ,\, \dots ,\, n+1}$ coincide on the intersections i.e.,
\[
  \omega_{r_1,\, \dots ,\, r_{n+1}}^i = \omega_{r_1,\, \dots ,\, r_{n+1}}^j \quad \text{on} \quad U_i \cap U_j.
\]
\end{clm}

Using these claims, we complete the proof of Lemma \ref{cohomologyn}.
Obviously, $\{ \omega_{r_1,\, \dots ,\, r_{n + 1}} \}$ are linearly independent, thus these form a basis of $H^0(\mathbb{CP}^n,\, K_n \otimes \mathcal{O}(-m))$ since we have proved that $\dim H^0(\mathbb{CP}^n,\, K_n \otimes \mathcal{O}(-m)) = \left( \substack{-m-1 \\ n} \right)$ in (\ref{dimcoh}).
Moreover since the action of any $u = (u_1,\, \dots ,\, u_{n + 1}) \in T^{n + 1}$ preserves $U_i$, we can find out its action on $\omega_{r_1,\, \dots ,\, r_n}$ in locally as follows:
\begin{align*}
  &\quad u \cdot \omega_{r_1,\, \dots ,\, r_n} \\[5pt]
  &= (-1)^{i + 1} \frac{u_i}{u_1} \dots \frac{u_i}{u_{i - 1}} \frac{u_i}{u_{i + 1}} \dots \frac{u_i}{u_{n+1}} 
    d\zeta_1^{(i)} \wedge \dots \wedge d\zeta_n^{(i)} \otimes 
      u_i^{- ( n + 1 )} u_1^{- r_1} \dots u_{n + 1}^{- r_{n + 1}} s_i^{n + 1} s_1^{r_1} \dots s_{n + 1}^{r_{n + 1}} \\[5pt]
  &= (-1)^{i+1} u_1^{-1} \dots u_{n+1}^{-1} u_1^{-r_1} \dots u_{n+1}^{-r_{n+1}} 
    d\zeta_1^{(i)} \wedge \dots \wedge d\zeta_n^{(i)} \otimes s_i^{n+1} s_1^{r_1} \dots s_{n+1}^{r_{n+1}} \\[5pt]
  &= u_1^{-1} \dots u_{n+1}^{-1} u_1^{-r_1} \dots u_{n+1}^{-r_{n+1}} \omega_{r_1,\, \dots ,\, r_{n+1}}
\end{align*}
on $U_i$.
Thus for each $(r_1 ,\, \dots r_{n+1}) \in \mathbb{Z}_{\geq 0}^{n+1}$ such that $r_1 \cdots r_{n+1} = -m -n -1$, the space $\langle \omega_{r_1,\, \dots ,\, r_{n+1}} \rangle$ which is spanned by $\omega_{r_1,\, \dots ,\, r_{n+1}}$ over $\mathbb{C}$ is a subrepresentation space of $H^0(\mathbb{CP}^n,\, K_n \otimes \mathcal{O}(-m))$ of $T^{n+1}$ and its character is $t_1^{-1} \dots t_{n+1}^{-1} t_1^{-r_1} \dots t_{n+1}^{-r_{n+1}}$. This implies that the character of $H^0(\mathbb{CP}^n,\, K_n \otimes \mathcal{O}(-m))$ is given by $\displaystyle t_1^{-1} \dots t_{n+1}^{-1} \sum_{\substack{r_1 + \dots + r_{n+1} = -m-n-1 \\ r_1,\, \dots ,\, r_{n+1} \geq 0}} t_1^{-r_1} \dots t_{n+1}^{-r_{n+1}}$ and 
\begin{align*}
  H^0(\mathbb{CP}^n,\, K_n \otimes \mathcal{O}(-m)) \cong 
    (\bigwedge\nolimits^{n+1} ( \mathbb{C}^{n+1} )^*) \otimes S^{-m-n-1}( (\mathbb{C}^{n+1})^* ) .
\end{align*}
Therefore from (\ref{cohnkey1}), we conclude that
\[
  H^n(\mathbb{CP}^n,\, \mathcal{O}(m)) \cong H^0(\mathbb{CP}^n,\, K_n \otimes \mathcal{O}(-m))^* \cong 
    \left( \bigwedge\nolimits^{n+1} \mathbb{C}^{n+1} \right) \otimes S^{-m-n-1}( \mathbb{C}^{n+1} )
\]
and its character is equal to $\displaystyle t_1 \dots t_{n+1} \sum_{\substack{r_1 + \dots + r_{n+1} = -m-n-1 \\ r_1,\, \dots ,\, r_{n+1} \geq 0}} t_1^{r_1} \dots t_{n+1}^{r_{n+1}}$.
\end{proof}

Finally, let us verify Claim \ref{transition} and \ref{gluing}.

\begin{proof}[Proof of Claim \ref{transition}]  According to general argument, we have
\[
  d\zeta_1^{(i)} \wedge \dots \wedge d\zeta_n^{(i)} = 
    \text{det} \frac{\partial \zeta^{(i)}}{\partial \zeta^{(j)}} \, d\zeta_1^{(j)} \wedge \dots \wedge d\zeta_n^{(j)} ,
\]
where $\text{det} \frac{\partial \zeta^{(i)}}{\partial \zeta^{(j)}}$ denotes the determinant of the Jacobian matrix.
Thus it is sufficient to see that
\[
  \text{det} \frac{\partial \zeta^{(i)}}{\partial \zeta^{(j)}} = 
    (-1)^{i+j}\frac{z_j^{n+1}}{z_i^{n+1}} \quad \text{on} \quad U_i \cap U_j.
\]
We can assume that $i < j$ without loss of generality. Since 
\begin{align*}
&(\zeta_1^{(i)},\, \dots ,\, \zeta_{i-1}^{(i)},\, \zeta_{i}^{(i)},\, \dots ,\, \zeta_{j-2}^{(i)},\, \zeta_{j-1}^{(i)},\, \zeta_{j}^{(i)},\, \dots ,\, \zeta_{n}^{(i)}) \\
&\qquad \qquad \qquad \qquad \qquad \qquad \qquad = \left( \frac{\zeta_1^{(j)}}{\zeta_i^{(j)}} ,\, \dots ,\, \frac{\zeta_{i-1}^{(j)}}{\zeta_i^{(j)}} ,\, \frac{\zeta_{i+1}^{(j)}}{\zeta_i^{(j)}} ,\, \dots ,\, \frac{\zeta_{j-1}^{(j)}}{\zeta_i^{(j)}} ,\, \frac{1}{\zeta_i^{(j)}} ,\, \frac{\zeta_j^{(j)}}{\zeta_i^{(j)}} ,\, \dots ,\, \frac{\zeta_n^{(j)}}{\zeta_i^{(j)}} \right)
\end{align*}
on the intersection $U_i \cap U_J$, from straightforward calculation, we have the Jacobian matrix as follows:

\[ \left( \setlength{\extrarowheight}{8pt}
\begin{array}{cccc|c|cccc|cccc}
\frac{1}{\zeta_i^{(j)}} & 0 & \dots & 0 & -\frac{\zeta_1^{(j)}}{(\zeta_i^{(j)})^2} & 0 & \hdotsfor{2} & 0 & 0 & \hdotsfor{2} & 0 \\
0 & \frac{1}{\zeta_i^{(j)}} & \ddots & \vdots & -\frac{\zeta_2^{(j)}}{( \zeta_i^{(j)} )^2} & \vdots & \ddots & & \vdots & \vdots & \ddots & & \vdots \\
\vdots & \ddots & \ddots & 0 & \vdots & \vdots & & \ddots & \vdots & \vdots & & \ddots &\vdots \\
0 & \dots & 0 & \frac{1}{\zeta_i^{(j)}} & -\frac{\zeta_{i-1}^{(j)}}{( \zeta_i^{(j)} )^2} & 0 & \hdotsfor{2} & 0 & 0 & 0 & \dots & 0 \\[8pt] \hline
0 & \hdotsfor{2} & 0 & -\frac{\zeta_{i+1}^{(j)}}{( \zeta_i^{(j)} )^2} & \frac{1}{\zeta_i^{(j)}} & 0 & \dots & 0 & 0 & \hdotsfor{2} & 0 \\
\vdots & \ddots & & \vdots & \vdots & 0 & \frac{1}{\zeta_i^{(j)}} & \ddots & \vdots & \vdots & \ddots & & \vdots \\
\vdots & & \ddots & \vdots & \vdots & \vdots & \ddots & \ddots & 0 & \vdots & & \ddots & \vdots \\
0 & \hdotsfor{2} & 0 & -\frac{\zeta_{j-1}^{(j)}}{( \zeta_i^{(j)} )^2} & 0 & \dots & 0 & \frac{1}{\zeta_i^{(j)}} & 0 & \hdotsfor{2} & 0 \\[8pt] \hline
0 & \hdotsfor{2} & 0 & -\frac{1}{( \zeta_i^{(j)} )^2} & 0 & \hdotsfor{2} & 0 & 0 & \hdotsfor{2} & 0 \\[8pt] \hline
0 & \hdotsfor{2} & 0 & -\frac{\zeta_{j}^{(j)}}{( \zeta_i^{(j)} )^2} & 0 & \hdotsfor{2} & 0 & \frac{1}{\zeta_i^{(j)}} & 0 & \dots & 0 \\
\vdots & \ddots & & \vdots & \vdots & \vdots & \ddots & & \vdots & 0 & \frac{1}{\zeta_i^{(j)}} & \ddots & \vdots \\
\vdots & & \ddots & \vdots & \vdots & \vdots & & \ddots & \vdots & \vdots & \ddots & \ddots & 0 \\
0 & \hdotsfor{2} & 0 & -\frac{\zeta_{n}^{(j)}}{( \zeta_i^{(j)} )^2} & 0 & \hdotsfor{2} & 0 & 0 & \dots & 0 & \frac{1}{\zeta_i^{(j)}}
\end{array}
\right) \]
Thus the determinant of the Jacobian matrix is equal to $(-1)^{i+j}\frac{z_j^{n+1}}{z_i^{n+1}}$.
\end{proof}

\begin{proof}[Proof of Claim \ref{gluing}]
It is sufficient to check that 
\[
  \left( \omega_{r_1,\, \dots ,\, r_{n+1}}^i (\bm{l}) \right) ( v^{(1)} \otimes \dots \otimes v^{(-m)} ) = 
    \left( \omega_{r_1,\, \dots ,\, r_{n+1}}^j (\bm{l}) \right) ( v^{(1)} \otimes \dots \otimes v^{(-m)} ) 
      \quad \in K_n|_{\bm{l}}
\]
for any $\bm{l} = [z_1 : \dots : z_{n+1}] \in U_i \cap U_j$, where $v^{(k)} \in \bm{l} \subset \mathbb{C}^{n+1}$ $( 1 \leq k \leq -m)$. 
By definition and Claim \ref{transition},
\begin{align*}
  & \quad \left( \omega_{r_1,\, \dots ,\, r_{n+1}}^i (\bm{l}) \right) ( v^{(1)} \otimes \dots \otimes v^{(-m)} ) \\[5pt]
  &= (-1)^{i+1} (d\zeta_1^{(i)})_{\bm{l}} \wedge \dots \wedge (d\zeta_n^{(i)})_{\bm{l}} \otimes 
    (s_i^{n+1} s_1^{r_1} \dots s_{n+1}^{r_{n+1}})(v^{(1)} \otimes \dots \otimes v^{(-m)}) \\[5pt]
  &= (-1)^{j+1} \left( \frac{z_j}{z_i} \right)^{n+1} (d\zeta_1^{(j)})_{\bm{l}} \wedge \dots \wedge (d\zeta_n^{(j)})_{\bm{l}} 
    \otimes (s_i^{n+1} s_1^{r_1} \dots s_{n+1}^{r_{n+1}})(v^{(1)} \otimes \dots \otimes v^{(-m)}).
\end{align*}
Moreover since $v^{(k)} \in \bm{l} \subset \mathbb{C}^{n+1}$ for each $k \in \{ 1,\, \dots ,\, -m \}$, we have $\frac{v_j^{(k)}}{v_i^{(k)}} = \frac{z_j}{z_i}$ where $v^{(k)} = (v_1^{(k)},\, \dots ,\, v_{n+1}^{(k)}) \in \mathbb{C}^{n+1}$. Using these relations, we find that
\[
  \left( \frac{z_j}{z_i} \right)^{n+1} (s_i^{n+1} s_1^{r_1} \dots s_{n+1}^{r_{n+1}})(v^{(1)} \otimes \dots \otimes v^{(-m)}) 
    = (s_j^{n+1} s_1^{r_1} \dots s_{n+1}^{r_{n+1}})(v^{(1)} \otimes \dots \otimes v^{(-m)}) \quad \in \mathbb{C}.
\]
Thus we obtain the required equation.
\end{proof}

\subsection{The transversal index of the lifted Dolbeault operator} \label{formula}

In this section, we give a proof of our main theorem. Firstly, let us prepare some notations.

\begin{dfn} \label{J}
For $a \in \mathbb{C} \setminus \{ 0 \}$, let us put
\begin{itemize}
  \item $\displaystyle \left[ \frac{1}{1-az} \right]^{z = 0} := 1 + az + a^2z^2 + a^3z^3 + \dotsb$ ,
  \item $\displaystyle \left[ \frac{1}{1-az} \right]^{z = \infty} := 
             -\left( a^{-1}z^{-1} + a^{-2}z^{-2} + a^{-3}z^{-3} + \dotsb \right)$ .
\end{itemize}
These are the Laurent expansions of $\displaystyle \frac{1}{1-az}$ at $z = 0$ and $\infty$ respectively. More generally, for any meromorphic function $\phi(z)$ on Riemann sphere, we denote the Laurent expansions of $\phi$ at $z=0$ and $z=\infty$ by $[\phi]^{z=0}$ and $[\phi]^{z=\infty}$ respectively.

Let $n \geq 0$.  We denote 
\[
  \lambda_{n+1}(k) := \prod_{j=1}^{n+1} (1-t_j^kt)  \quad \in R(T^{n+1} \times S^1).
\]
\begin{align*}
  J_0(\lambda_{n+1}(k)) &:= \left[ \frac{1}{\lambda_{n+1}(k)} \right]^{t = 0} = 
    \prod_{j=1}^{n+1} \left[ \frac{1}{1-t_j^kt} \right]^{t = 0} = 
      \prod_{j=1}^{n+1} \left( 1 + t_j^kt + t_j^{2k}t^2 + t_j^{3k}t^3 + \dotsb \right) , \\[5pt]
  J_{\infty}(\lambda_{n+1}(k)) &:= \left[ \frac{1}{\lambda_{n+1}(k)} \right]^{t = \infty} = 
    \prod_{j=1}^{n+1} \left[ \frac{1}{1-t_j^kt} \right]^{t = \infty} = 
      (-1)^{n+1} \prod_{j=1}^{n+1} \left( t_j^{-k}t^{-1} + t_j^{-2k}t^{-2} + \dotsb \right) , \\[5pt]
  J(\lambda_{n+1}(k)) &:= J_0(\lambda_{n+1}(k)) - J_{\infty}(\lambda_{n+1}(k)) = 
    \prod_{j=1}^{n+1} \left[ \frac{1}{1-t_j^kt} \right]^{t = 0} - \prod_{j=1}^{n+1} \left[ \frac{1}{1-t_j^kt} \right]^{t = \infty} .
\end{align*}
\end{dfn}

\begin{rem} \label{lamandJ}
\begin{enumerate}
  \item $\lambda_{n+1}(k)$ are regarded as the characters of the exterior algebras of $\mathbb{C}^{n+1}$ provided with the action of $T^{n+1} \times S^1$ whose weight is $((\underbrace{k,\, \dots ,\, k}_{n+1}),\, 1)$.

  \item $J_0(\lambda_{n+1}(k))$ and $J_{\infty}(\lambda_{n+1}(k))$ are the Fourier expansions of distributions on $T^{n+1} \times S^1$ which can be regarded as linear maps from $\mathcal{D}(S^1)$ to $\mathcal{D}(T^{n+1})$ since the coefficients of $t^m$ are Laurent polynomials in variables $t_1^k, \, \dots ,\, t_{n+1}^k$ (see also Remark \ref{spform}).

  \item Since $J_0(\lambda_{n+1}(k))$ and $J_{\infty}(\lambda_{n+1}(k))$ are the Laurent expansions of $\lambda_{n+1}(k)^{-1}$ at $t=0$ and $t=\infty$ respectively, $\lambda_{n+1}(k) J_0(\lambda_{n+1}(k)) = \lambda_{n+1}(k) J_{\infty}(\lambda_{n+1}(k)) =1$. \label{J}
In particular, we have $\lambda_{n+1}(k) J(\lambda_{n+1}(k)) = 0$.

  \item We sometimes denote $J_0(\lambda_{n+1}(k))$ by $J_0(\lambda_{n+1}(k))(t_1, \dots ,\, t_{n+1} ,\, t)$ to emphasize that it is the Fourier expansion with respect to a coordinate $(t_1, \, \dots ,\, t_{n+1}, \, t)$ of $T^{n+1} \times S^1$. The similar notation is used for $J_{\infty}(\lambda_{n+1}(k))$.
\end{enumerate}
\end{rem}

Let us see the following fundamental relations of $\chi_{n,\,l}$.
\begin{lem} \label{shift}
For any integer $n \geq 1$ , we have
\begin{enumerate}
  \item \label{oneshift} $\chi_{n,\, l} = \chi_{n-1,\, l} + \chi_{n,\, l-1} \, t_{n+1}^{-1} \quad (l \in \mathbb{Z})$ ,
  \item \label{mshift} $\chi_{n,\, l} = \displaystyle \sum_{i=0}^{j-1} \chi_{n-1,\, l-i} \, t_{n+1}^{-i} + 
                                \chi_{n,\, l-j} \, t_{n+1}^{-j} \quad (l \in \mathbb{Z} \, , \, j \geq 1)$ ,
  \item \label{pshift} $\chi_{n,\, l} = \displaystyle \sum_{i=0}^l \chi_{n-1,\, l-i} \, t_{n+1}^{-i} \quad (l \geq 0)$ ,
  \item \label{nshift} $\chi_{n,\, l} = -\displaystyle \sum_{i=1}^{-l-1} \chi_{n-1, \, l+i} \, t_{n+1}^i \quad (l \leq -2)$ .
\end{enumerate}
\end{lem}

\begin{proof}
It is easy to prove them from Corollary \ref{chicharacter}. Indeed, (\ref{oneshift}) immediately follows from Corollary \ref{chicharacter}. (\ref{mshift}) is verified by using (\ref{oneshift}) repeatedly.
In addition, if we set $j=l+1$ in (\ref{mshift}), then we get (\ref{pshift}) because $\chi_{n,\, -1} = 0$.
Let $m \leq -2$. Setting $l = -1,\, j = -m-1$ in (\ref{pshift}), then
\[
  \chi_{n,\,-1} = \sum_{i=0}^{-m-2} \chi_{n-1,\, -1-i} \, t_{n+1}^{-i} + \chi_{n,\, m} \, t_{n+1}^{m+1}.
\]
Multiplying the both side by $t_{n+1}^{-m-1}$,
\[
  \chi_{n,\, m} = -\sum_{i=0}^{-m-2} \chi_{n-1,\, -1-i} \, t_{n+1}^{-i-1-m} = -\sum_{j=1}^{-m-1} \chi_{n,\, m+j} \, t_{n+1}^j.
\]
Thus we obtain (\ref{nshift}).
\end{proof}

\begin{lem} \label{recurrence}
We assume that $n \geq 1 ,\, k < 0$. Then,
\begin{enumerate}
  \item \label{prec}
    $\begin{aligned}[t]
      &\quad \sum_{m\geq 0} \chi_{n, \, -km} \, t^m \\
      &=\sum_{m \geq 0} \chi_{n-1,\, -km} \, t^m + 
         \left( \sum_{j=0}^{|k|-1} \left( \sum_{m\geq 0} \chi_{n-1, \, -km+j} \, t^m \right) t_{n+1}^{k+j}t \right) 
           \left[ \frac{1}{1-t_{n+1}^kt} \right]^{t = 0} \\[10pt]
      &= \left( (1-t_{n+1}^kt) \sum_{m \geq 0} \chi_{n-1, \, -km} t^m  + 
         \sum_{j=0}^{|k|-1} \left( \sum_{m\geq 0} \chi_{n-1, \, -km+j} \, t^m \right) t_{n+1}^{k+j}t \right) 
           \left[ \frac{1}{1-t_{n+1}^kt} \right]^{t = 0} \\[10pt]
      &= \left\{ 
        \begin{aligned} 
          &\left( \sum_{m \geq 0} \chi_{n-1, \, -km} t^m + 
            \sum_{j=1}^{|k|-1} \left( \sum_{m \geq 0} \chi_{n-1, \, -km+j} \, t^m \right) t_{n+1}^{k+j}t \right) 
            \left[ \frac{1}{1-t_{n+1}^kt} \right]^{t = 0} &\, &(k \leq -2) \, ,\\
          & \\
          & \left( \sum_{m \geq 0} \chi_{n-1, \, -km} t^m \right) \left[ \frac{1}{1-t_{n+1}^kt} \right]^{t = 0} &\, &(k = -1).
        \end{aligned} \right.
\end{aligned}$
\vspace{5mm}
  \item \label{nrec}
    $\begin{aligned}[t] 
      &\quad \sum_{m < 0} \chi_{n, \, -km} \, t^m \\
      &=\sum_{m < 0} \chi_{n-1,\, -km} \, t^m + 
           \left( \sum_{j=0}^{|k|-1} \left( \sum_{m < 0} \chi_{n-1, \, -km+j} \, t^m \right) t_{n+1}^{k+j}t \right) 
             \left[ \frac{1}{1-t_{n+1}^kt} \right]^{t = \infty} \\[10pt]
      &= \left( (1-t_{n+1}^kt) \sum_{m < 0} \chi_{n-1, \, -km} t^m  + 
           \sum_{j=0}^{|k|-1} \left( \sum_{m < 0} \chi_{n-1, \, -km+j} \, t^m \right) t_{n+1}^{k+j}t \right) 
             \left[ \frac{1}{1-t_{n+1}^kt} \right]^{t = \infty} \\[10pt]
      &= \left\{ 
            \begin{aligned} &\left( \sum_{m < 0} \chi_{n-1, \, -km} t^m + 
                   \sum_{j=1}^{|k|-1} \left( \sum_{m < 0} \chi_{n-1, \, -km+j} \, t^m \right) t_{n+1}^{k+j}t \right) 
                     \left[ \frac{1}{1-t_{n+1}^kt} \right]^{t = \infty} &\, &(k \leq -2) \, , \\
                  & \\
                  & \left( \sum_{m < 0} \chi_{n-1, \, -km} t^m \right) \left[ \frac{1}{1-t_{n+1}^kt} \right]^{t = \infty} &\, &(k = -1).
            \end{aligned} \right.
\end{aligned}$
\end{enumerate}
\end{lem}


\begin{proof}
Let us verify (\ref{prec}). Set $k = -k_0$. From (\ref{pshift}) in Lemma \ref{shift},
\begin{align*}
 \sum_{m \geq 0} \chi_{n, \, k_0m} \, t^m
 &= \sum_{m\geq 0}
 \left( \sum_{i=0}^{k_0m} \chi_{n-1, \, k_0m-i} \, t_{n+1}^{-i} \right)t^m \\
 &= \sum_{m \geq 0} \chi_{n-1,\, k_0m} t^m + \sum_{m \geq 1} \left( \sum_{i=1}^{k_0m} \chi_{n-1, \, k_0m-i} \, t_{n+1}^{-i} \right)t^m \\
 &= \sum_{m \geq 0} \chi_{n-1, \, k_0m} \, t^m + \sum_{m\geq 1} \left( \sum_{j=0}^{k_0-1} \sum_{\substack{i \equiv j \, mod \, k_0 \\ 1 \leq i \leq k_0m}} \chi_{n-1, \, k_0m-i} \, t_{n+1}^{-i} \right) t^m \\
 &= \sum_{m \geq 0} \chi_{n-1, \, k_0m} \, t^m + \sum_{m\geq 1} \left( \sum_{j=0}^{k_0-1} \sum_{\alpha = 1}^m\chi_{n-1,\, k_0m-(k_0\alpha -j)} \, t_{n+1}^{-(k_0\alpha -j)} \right) t^m \\
 &= \sum_{m \geq 0} \chi_{n-1, \, k_0m} \, t^m + \sum_{j=0}^{k_0-1} \left( \sum_{m\geq 1} \sum_{\alpha =1}^m \chi_{n-1, \, k_0(m-\alpha) + j} \, t^{m-\alpha} \, (t_{n+1}^{-k_0} t)^{\alpha -1} \right)  t_{n+1}^{-k_0+j} t \\
 &= \sum_{m \geq 0} \chi_{n-1, \, k_0m} \, t^m + \sum_{j=0}^{k_0-1} \left( \left( \sum_{l\geq 0} \chi_{n-1, \, k_0l + j}\, t^l \right) \left( \sum_{\alpha \geq 0} \, (t_{n+1}^{-k_0} t)^{\alpha} \right) \right) t_{n+1}^{-k_0+j}t. \\
 &= \sum_{m \geq 0} \chi_{n-1, \, k_0m} \, t^m + \left( \sum_{j=0}^{k_0-1} \left( \sum_{l\geq 0} \chi_{n-1, \, k_0l + j}\, t^l \right) t_{n+1}^{-k_0+j}t \right) \left( \sum_{\alpha \geq 0} \, (t_{n+1}^{-k_0} t)^{\alpha} \right) \\
 &= \sum_{m \geq 0} \chi_{n-1, \, k_0m} \, t^m + \left( \sum_{j=0}^{k_0-1} \left( \sum_{l \geq 0} \chi_{n-1, \, k_0l + j}\, t^l \right) t_{n+1}^{-k_0+j}t \right) \left[ \frac{1}{1-t_{n+1}^{-k_0}t} \right]^{t = 0}
\end{align*}
 
In addition, since $\displaystyle 1 = (1 - t_{n+1}^{-k_0}t) \left[ \frac{1}{1-t_{n+1}^{-k_0}t} \right]^{t = 0}$, we have
\begin{align*}
  \sum_{m \geq 0} \chi_{n, \, k_0m} \, t^m
  &= \sum_{m \geq 0} \chi_{n-1, \, k_0m} \, t^m + \left( \sum_{j=0}^{k_0-1} \left( \sum_{l \geq 0} \chi_{n-1, \, k_0l + j}\, t^l \right) t_{n+1}^{-k_0+j}t \right) \left[ \frac{1}{1-t_{n+1}^{-k_0}t} \right]^{t = 0} \\
  &= \left( (1 - t_{n+1}^{-k_0}t) \sum_{m \geq 0} \chi_{n-1, \, k_0m} \, t^m + \sum_{j=0}^{k_0-1} \left( \sum_{l \geq 0} \chi_{n-1, \, k_0l + j}\, t^l \right) t_{n+1}^{-k_0+j}t \right)  \left[ \frac{1}{1-t_{n+1}^{-k_0}t} \right]^{t = 0}.
\end{align*}

Similarly, we can prove (\ref{nrec}) by using (\ref{nshift}) in Lemma \ref{shift}.
\end{proof}

\begin{lem} \label{A}
For $n \geq 0$, \, $k < 0$, \, $0 \leq j \leq |k|-1$, there exists a unique function $A_{n,\, k}^j$ on $T^{n+1} \times S^1$ such that 
\begin{enumerate}
  \item \label{A+} $\displaystyle \sum_{m \geq 0} \chi_{n,\, -km+j} \, t^m = A_{n,\, k}^j \, J_0(\lambda_{n+1}(k))$ ,
  \item \label{A-} $\displaystyle \sum_{m < 0} \chi_{n,\, -km+j} \, t^m = -A_{n,\, k}^j \, J_{\infty}(\lambda_{n+1}(k))$ .
\end{enumerate}
\end{lem}

\begin{proof}
The uniqueness of $A_{n,\, k}^j$ follows from the equation that $\lambda_{n+1}(k) \, J_{0}(\lambda_{n+1}(k)) = 1$ (see Remark \ref{lamandJ} (\ref{J})).
Let us verify the existence of $A_{n,\, k}^j$.
If $k = -1$ and $j = 0$, by using Lemma \ref{recurrence} repeatedly, we have 
\begin{align*}
\sum_{m \geq 0} \chi_{n,\, m} \, t^m &= \left( \sum_{m \geq 0} \chi_{n-1,\, m} \, t^m \right) \left[ \frac{1}{1-t_{n+1}^{-1}t} \right]^{t = 0} \\
&= \dots = \left( \sum_{m \geq 0} \chi_{0,\, m} \, t^m \right) \left[ \frac{1}{1-t_{2}^{-1}t} \right]^{t = 0} \dotsm \left[ \frac{1}{1-t_{n+1}^{-1}t} \right]^{t = 0}.
\end{align*}
Here, since
\[
\sum_{m \geq 0} \chi_{0,\, m} \, t^m
= \sum_{m \geq 0} t_1^{-m} t^m = \left[ \frac{1}{1-t_{1}^{-1}t} \right]^{t = 0},
\]
we obtain
\begin{align*}
\sum_{m \geq 0} \chi_{n,\, m} \, t^m &= J_0(\lambda_{n+1}(-1)).
\end{align*}
Analogously, we have
\[
\sum_{m < 0} \chi_{n,\, m} \, t^m = -J_{\infty}(\lambda_{n+1}(-1)).
\]
Thus we can find $A_{n,\, -1}^0 = 1$.

Next let $k \leq -2$ be fixed. We will construct $A_{n,\, k}^j$ by the induction on $n$ and $j$.
If $n=0$, it is easy to construct $A_{0,\, k}^j$.
Indeed, since $\chi_{0,\,l} = t_1^{-l}$,
\begin{align*}
  \sum_{m \geq 0} \chi_{0,\, -km+j} \, t^m &= \sum_{m \geq 0} t_1^{-(-km+j)} t^m
    = t_1^{-j} \sum_{m \geq 0} (t_1^{k}t)^m = t_1^{-j} \left[ \frac{1}{1-t_1^{k}t} \right]^{t = 0}
      = t_1^{-j} J_0(\lambda_{1}(k)) ,
\end{align*}
and
\begin{align*}
  \sum_{m < 0} \chi_{0,\, -km+j} \, t^m &= \sum_{m < 0} t_1^{-(-km+j)} t^m 
    = t_1^{-j} \sum_{m < 0} (t_1^{k}t)^m = -t_1^{-j} \left[ \frac{1}{1-t_1^{k}t} \right]^{t = \infty} 
      = t_1^{-j} J_{\infty}(\lambda_{1}(k)).
\end{align*}
Thus we find $A_{0,\, k}^j = t_1^{-j}$ for $0 \leq j \leq |k|-1$.
Let $n \geq 1$ and we assume that there exist $A_{n-1,\, k}^j$ for $0 \leq j \leq |k|-1$ as in the Lemma.
From Lemma \ref{recurrence} and the assumption of the induction,
\begin{align*}
  \sum_{m\geq 0} \chi_{n, \, -km} \, t^m
  &= \left( \sum_{m \geq 0} \chi_{n-1, \, -km} t^m + 
      \sum_{j=1}^{|k|-1} \left( \sum_{m \geq 0} \chi_{n-1, \, -km+j} \, t^m \right) t_{n+1}^{k+j}t \right) 
      \left[ \frac{1}{1-t_{n+1}^{k}t} \right]^{t = 0} \\
  &= \left( A_{n-1,\, k}^0 \,J_0(\lambda_{n}(k)) + 
      \sum_{j=1}^{|k|-1} A_{n-1,\, k}^j \, t_{n+1}^{k+j}t \,J_0(\lambda_{n}(k)) \right) 
      \left[ \frac{1}{1-t_{n+1}^{k}t} \right]^{t = 0} \\
  &= \left( A_{n-1,\, k}^0 + \sum_{j=1}^{|k|-1} A_{n-1,\, k}^j \, t_{n+1}^{k+j}t \right) J_0(\lambda_{n+1}(k))
\end{align*}
and
\begin{align*}
  \sum_{m < 0} \chi_{n, \, -km} \, t^m
  &= \left( \sum_{m < 0} \chi_{n-1, \, -km} t^m + 
      \sum_{j=1}^{|k|-1} \left( \sum_{m < 0} \chi_{n-1, \, -km+j} \, t^m \right) t_{n+1}^{k+j}t \right) 
      \left[ \frac{1}{1-t_{n+1}^{k}t} \right]^{t = \infty} \\
  &= -\left( A_{n-1,\, k}^0 \,J_0(\lambda_{n}(k)) + 
      \sum_{j=1}^{|k|-1} A_{n-1,\, k}^j \, t_{n+1}^{k+j}t \,J_{\infty}(\lambda_{n}(k)) \right) 
      \left[ \frac{1}{1-t_{n+1}^{k}t} \right]^{t = 0} \\
  &= -\left( A_{n-1,\, k}^0 + \sum_{j=1}^{|k|-1} A_{n-1,\, k}^j \, t_{n+1}^{k+j}t \right) J_{\infty}(\lambda_{n+1}(k)).
\end{align*}
Therefore we can find that $\displaystyle A_{n,\, k}^0 = A_{n-1,\, k}^0 + \sum_{j=1}^{|k|-1} A_{n-1,\, k}^j \, t_{n+1}^{k+j}t$.
Moreover from (\ref{oneshift}) in Lemma \ref{shift} and the assumption of the induction,
\begin{align*}
  \sum_{m \geq 0} \chi_{n,\, -km+j} \, t^m &= \sum_{m \geq 0} \chi_{n-1,\, -km+j} \, t^m + 
            t_{n+1}^{-1} \sum_{m \geq 0} \chi_{n,\, -km+j-1} \, t^m \\
  &= A_{n-1,\, k}^{j} \, J_0(\lambda_{n}(k)) + t_{n+1}^{-1} A_{n,\, k}^{j-1} \, J_0(\lambda_{n+1}(k)) \\
  &= \left( (1-t_{n+1}^kt) A_{n-1,\, k}^{j} + t_{n+1}^{-1} A_{n,\, k}^{j-1} \right) \, J_0(\lambda_{n+1}(k))
\end{align*}
and
\begin{align*}
  \sum_{m < 0} \chi_{n,\, -km+j} \, t^m &= \sum_{m < 0} \chi_{n-1,\, -km+j} \, t^m + 
            t_{n+1}^{-1} \sum_{m < 0} \chi_{n,\, -km+j-1} \, t^m \\
  &= -A_{n-1,\, k}^{j} \, J_{\infty}(\lambda_{n}(k)) - t_{n+1}^{-1} A_{n,\, k}^{j-1} \, J_{\infty}(\lambda_{n+1}(k)) \\
  &= -\left( (1-t_{n+1}^kt) A_{n-1,\, k}^{j} + t_{n+1}^{-1} A_{n,\, k}^{j-1} \right) \, J_{\infty}(\lambda_{n+1}(k)).
\end{align*}
Thus for $1 \leq j \leq |k|-1$, we can find that $A_{n,\, k}^j = (1-t_{n+1}^kt) A_{n-1,\, k}^{j} + t_{n+1}^{-1} A_{n,\, k}^{j-1}$ inductively.
\end{proof}

\begin{rem} \label{Arecurrence}
From the above proof, we find that for $k \leq -2$, $A_{n,\, k}^j$ satisfy the following relations:
\begin{align}
    &A_{n,\, k}^0 = A_{n-1,\, k}^0 + \sum_{j=1}^{|k|-1} A_{n-1,\, k}^j \, t_{n+1}^{k+j}t \label{A0} \\
    &A_{n,\, k}^j = (1-t_{n+1}^kt) A_{n-1,\, k}^{j} + t_{n+1}^{-1} A_{n,\, k}^{j-1} \label{Aj} \qquad (1 \leq j \leq |k|-1)
\end{align}
for $n \geq 1$, and $A_{0,\, k}^j = t_1^{-j}$ $(0 \leq j \leq |k|-1)$. We also find that if $k = -1$, $A_{n,\, -1}^0 = 1$ for all $n \geq 0$.
\end{rem}

The following is a key lemma.

\begin{lem} \label{chiandsym}
Let $n\geq 0$,\, $k \in \mathbb{Z}$,\, $j \in \mathbb{Z}$. Then the following equation holds:
\[
  \sum_{i=0}^{n+1} (-1)^i \chi_{n,\, ik+j}  \, s_i (t_1^k,\, \dots, \, t_{n+1}^k) = 0 ,
\]
where $s_i (t_1^k,\, \dots, \, t_{n+1}^k)$ are the $i$-th elementary symmetric polynomials in variables $t_1^k,\, \dots, \, t_{n+1}^k$.
\end{lem}

\begin{proof}
\begin{itemize}
  \item $k<0$ : For any $j \in \mathbb{Z}$, there exist integers $\alpha$ and $j_0$ 
             such that $j=-k\alpha + j_0$, and $0 \leq j_0 \leq |k|-1$.
  From Lemma \ref{A}, we have
    \[
       \sum_{m \in \mathbb{Z}} \chi_{n,\, -km+j_0} \, t^m = A_{n,\, k}^{j_0} \,J(\lambda_{n+1}(k)).
    \]
  In addition, since $\lambda_{n+1}(k)J(\lambda_{n+1}(k)) = 0$,
    \[
      \lambda_{n+1}(k)\sum_{m \in \mathbb{Z}} \chi_{n,\, -km+j_0} \, t^m =0.
    \]
  Calculating the coefficient of $t^{\alpha}$ in the left hand side, we find that
    \[
      \sum_{i=0}^{n+1} (-1)^i \chi_{,\, ik-k\alpha +j_0} \, s_i(t_1^k,\, \dots ,\, t_{n+1}^k) = 0.
    \]
  Thus we obtain the required equation for $k<0$.

  \item $k=0$ : For any $j \in \mathbb{Z}$, we have
    \[
      \sum_{i=0}^{n+1} (-1)^i \chi_{n,\, j} \, s_i(t_1^0,\, \dots ,\, t_{n+1}^0) = 
        \chi_{n,\, j} \sum_{i=0}^{n+1} (-1)^i \Big( \substack{n+1 \\ \\ i} \Big) = \chi_{n,\, j} (1-1)^{n+1} =0.
    \]

  \item $k>0$ : Let $j \in \mathbb{Z}$. Since $s_{\alpha} (t_1^k,\, \dots, \, t_{n+1}^k) = t_1^k \dots t_{n+1}^k s_{n+1-\alpha} (t_1^{-k},\, \dots, \, t_{n+1}^{-k})$ for \, $0 \leq \alpha \leq n+1$,
    \begin{align*}
      &\quad \sum_{\alpha =0}^{n+1} (-1)^{\alpha} \chi_{n,\, \alpha k+j}  \, s_{\alpha} (t_1^k,\, \dots, \, t_{n+1}^k) \\
      &= t_1^k \dots t_{n+1}^k \sum_{\alpha =0}^{n+1} (-1)^{\alpha} \chi_{n,\, \alpha k+j}  \, 
             s_{n+1-\alpha} (t_1^{-k},\, \dots, \, t_{n+1}^{-k})\\
      &= (-1)^{n+1} t_1^k \dots t_{n+1}^k \sum_{i =0}^{n+1} (-1)^{i} \chi_{n,\, i(-k) + k(n+1)+j}  \, 
             s_i (t_1^{-k},\, \dots, \, t_{n+1}^{-k}) \\
      &=0.
    \end{align*}
  The last equality follows from the case of $k < 0$.
\end{itemize}
\end{proof}

\begin{dfn}
For $n \geq 0$, \, $k \in \mathbb{Z}$, and $0\leq j \leq \max\{0, \, |k|-1\}$, we define the smooth functions $B_{n,\, k}^{j}$ on $T^{n+1} \times S^1$ as follows:
\begin{align*}
  B_{n,\, k}^j &= B_{n,\, k}^j (t_1,\, \dots ,\, t_{n+1}, \, t) := 
    \sum_{l=0}^n \left( \sum_{i=1}^{n+1-l} (-1)^{i+l+1} \chi_{n,\, ik+j} \, s_{i+l}(t_1^k,\, \dots, \, t_{n+1}^k) \right) t^l,
\end{align*}
where $s_{i+l} (t_1^k,\, \dots, \, t_{n+1}^k)$ are the $(i+l)$-th elementary symmetric polynomials in variables $t_1^k,\, \dots, \, t_{n+1}^k$.
\end{dfn}

\begin{rem} \label{anoB}
From the above Lemma \ref{chiandsym}, $B_{n,\, k}^j$ can be expressed in another way as follows:
\begin{align*}
  B_{n,\, k}^j &= \sum_{l=0}^n \left( \sum_{i=1}^{n+1-l} (-1)^{i+l+1} 
      \chi_{n,\, ik+j} \, s_{i+l} (t_1^k,\, \dots, \, t_{n+1}^k) \right) t^l \\
  &= \sum_{l=0}^n \left( \sum_{\alpha =0}^l (-1)^{\alpha} 
      \chi_{n,\, (\alpha -l)k+j} \, s_{\alpha}(t_1^k,\, \dots, \, t_{n+1}^k) \right) t^l
\end{align*}
\end{rem}

\begin{lem} \label{Brecurrence}
Let $k \leq -2$ be fixed. For $n \geq 0$, $B_{n,\, k}^j$ satisfy the following relations:
  \begin{align}
    & B_{n,\, k}^0 = B_{n-1,\, k}^0 + \sum_{j = 1}^{|k|-1} B_{n-1,\, k}^j t_{n+1}^{k+j} t , \label{B0} \\
    & B_{n,\, k}^j = (1-t_{n+1}^kt) B_{n-1,\, k}^j + t_{n+1}^{-1} B_{n,\, k}^{j-1} \qquad (1 \leq j \leq |k|-1) . \label{Bj}
  \end{align}
\end{lem}

\begin{proof}
By definition,
\begin{align}
  B_{n-1,\, k}^0 + \sum_{j=1}^{|k|-1} B_{n-1,\, k}^j \, t_{n+1}^{k+j}t
  &= \sum_{l=0}^{n-1} \left( \sum_{i=1}^{n-l} (-1)^{i+l+1} \chi_{n-1,\, ik+j} \, s_{i+l}(t_1^k,\, \dots, \, t_{n}^k) \right) t^l \notag \\
  &\qquad \qquad + \sum_{j=1}^{|k|-1} \left( \sum_{l=0}^{n-1} \left( \sum_{i=1}^{n-l} 
       (-1)^{i+l+1} \chi_{n-1,\, ik} \, s_{i+l}(t_1^k,\, \dots, \, t_{n}^k) \right) t^l \right) t_{n+1}^{k+j}t. 
\end{align}
From Lemma \ref{chiandsym} and Corollary \ref{chicharacter}, the coefficient of $t^0$ in the above equation is as follows:
\[
  \sum_{i=1}^{n} (-1)^{i+1} \chi_{n-1,\, ik} \, s_{i}(t_1^k,\, \dots, \, t_{n}^k) = 
    \chi_{n-1,\, 0} = 1 = \chi_{n,\, 0} = \sum_{i=1}^{n+1} (-1)^{i+1} \chi_{n,\, ik} \, s_{i}(t_1^k,\, \dots, \, t_{n+1}^k).
\]
In addition, the coefficient of $t^n$ is 
\begin{align*}
  \sum_{j=1}^{|k|-1} (-1)^{n-1} \chi_{n-1,\, k+j} \, s_n(t_1^k,\, \dots ,\, t_n^k) t_{n+1}^{k+j} 
  &= (-1)^{n-1} \left( \sum_{j=1}^{|k| - 1}  \chi_{n-1,\, k+j} \, t_{n+1}^j \right)  s_{n+1}(t_1^k,\, \dots ,\, t_{n+1}^k) \\
  &= (-1)^n \chi_{n,\, k} s_{n+1}(t_1^k,\, \dots ,\, t_{n+1}^k).
\end{align*}
Here we used (\ref{nshift}) in Lemma \ref{shift}.
If $n \geq 2$, the coefficients of $t^l$ $(1 \leq l \leq n-1)$ are calculated as follows:
\begin{align} \label{coef}
  & \quad \sum_{i=1}^{n-l} (-1)^{i+l+1} \chi_{n-1,\, ik} \, s_{i+l}(t_1^k,\, \dots, \, t_{n}^k) + 
       \sum_{j=1}^{|k|-1} \sum_{i=1}^{n+1-l} (-1)^{i+l} \chi_{n-1,\, ik} \, s_{i+l-1}(t_1^k,\, \dots, \, t_{n}^k) t_{n+1}^{k+j} \notag \\
  &=  \sum_{i=1}^{n-l} (-1)^{i+l+1} \chi_{n-1,\, ik} \, s_{i+l}(t_1^k,\, \dots, \, t_{n}^k) + 
       \sum_{i=1}^{n+1-l} (-1)^{i+l}\left(  \sum_{j=1}^{|k|-1} 
           \chi_{n-1,\, ik} \, t_{n+1}^{k+j} \right) \, s_{i+l-1}(t_1^k,\, \dots, \, t_{n}^k).
\end{align}
Here from (\ref{mshift}) in Lemma \ref{shift}, we have
\[
  \sum_{j=1}^{|k|-1} \chi_{n-1,\, ik} \, t_{n+1}^{k+j} = \chi_{n,\, (i-1)k} - \chi_{n,\, ik} \, t_{n+1}^k - \chi_{n-1,\, (i-1)k}.
\]
Using this equation, $(\ref{coef})$ is deformed as follows:
\begin{align*}
  & \quad \sum_{i=1}^{n-l} (-1)^{i+l+1} \chi_{n-1,\, ik} \, s_{i+l}(t_1^k,\, \dots, \, t_{n}^k) + 
      \sum_{i=1}^{n+1-l} (-1)^{i+l}\left(  \sum_{j=1}^{|k|-1} 
        \chi_{n-1,\, ik} \, t_{n+1}^{k+j} \right) \, s_{i+l-1}(t_1^k,\, \dots, \, t_{n}^k) \\
  &= (-1)^l s_{l}(t_1^k,\, \dots, \, t_{n}^k) + \sum_{i=1}^{n+1-l} (-1)^{i+l} \chi_{n,\, (i-1)k} \, s_{i+l-1}(t_1^k,\, \dots, \, t_{n}^k) \\
  &\qquad \qquad \qquad \qquad \qquad \qquad \qquad \qquad \qquad+ 
      \sum_{i=1}^{n+1-l} (-1)^{i+l+1} \chi_{n,\, ik} \, s_{i+l-1}(t_1^k,\, \dots, \, t_{n}^k) \, t_{n+1}^k \\
  &= \sum_{i=1}^{n+1-l} (-1)^{i+l+1} \chi_{n,\, ik} \, \left( s_{i+l}(t_1^k,\, \dots, \, t_{n}^k) + 
        s_{i+l-1}(t_1^k,\, \dots, \, t_{n}^k) \, t_{n+1}^k \right) \\
  &= \sum_{i=1}^{n+1-l} (-1)^{i+l+1} \chi_{n,\, ik} \, s_{i+l}(t_1^k,\, \dots, \, t_{n+1}^k).
\end{align*}
Therefore we have equation (\ref{B0}).
Next by definition and (\ref{shift}) in (\ref{oneshift}) in Lemma \ref{shift},
\begin{align*}
  B_{n,\, k}^{j} &= \sum_{l=0}^{n} \left( \sum_{i=1}^{n+1-l} 
    (-1)^{i+l+1} \chi_{n,\, ik+j} \, s_{i+l}(t_1^k,\, \dots, \, t_{n+1}^k) \right) t^l \\
  &= \sum_{l=0}^{n} \left( \sum_{i=1}^{n+1-l} 
     (-1)^{i+l+1} \chi_{n-1,\, ik+j} \, s_{i+l}(t_1^k,\, \dots, \, t_{n+1}^k) \right) t^l + t_{n+1}^{-1} B_{n,\, k}^{j-1}.
\end{align*}
In order to prove the equation (\ref{Bj}), it is sufficient to check the following:
\begin{equation} \label{B_n-1}
  \sum_{l=0}^{n} \left( \sum_{i=1}^{n+1-l} (-1)^{i+l+1} \chi_{n-1,\, ik+j} \, s_{i+l}(t_1^k,\, \dots, \, t_{n+1}^k) \right) t^l 
    = (1-t_{n+1}^kt) B_{n-1,\, k}^{j}.
\end{equation}
However this is very easy. Indeed, since
\[
  s_{i+l}(t_1^k,\, \dots, \, t_{n+1}^k) = s_{i+l}(t_1^k,\, \dots, \, t_{n}^k) + s_{i+l-1}(t_1^k,\, \dots, \, t_{n}^k) \, t_{n+1}^k,
\]
the left hand side in (\ref{B_n-1}) is equal to
\begin{align*}
  & \sum_{l=0}^{n} \left( \sum_{i=1}^{n+1-l} (-1)^{i+l+1} \chi_{n-1,\, ik+j} \, s_{i+l}(t_1^k,\, \dots, \, t_{n}^k) \right) t^l \\
  & \qquad \qquad \qquad \qquad \qquad \qquad + 
    t_{n+1}^k \sum_{l=0}^{n} \left( \sum_{i=1}^{n+1-l} (-1)^{i+l+1} \chi_{n-1,\, ik+j} \, s_{i+l-1}(t_1^k,\, \dots, \, t_{n}^k) \right) t^l.
\end{align*}
The first term is equal to $B_{n-1,\, k}^j$ because $s_{n+1}(t_1^k, \, \dots ,\, t_n^k) = 0$. Moreover from Lemma \ref{chiandsym}, we know that the coefficient of $t^0$ in the second term in the above is $0$. Thus by setting $l-1 = \tilde{l}$, the second term is equal to $-t_{n+1}^kt \, B_{n-1,\, k}^{j}$.
Therefore we obtain the equation (\ref{Bj}).
\end{proof}

\begin{prop} \label{A=B}
For $n \geq 0$, \, $k < 0$, and $0 \leq j \leq \max\{ 0,\,  |k|-1 \}$, we have 
\[
  A_{n,\, k}^j = B_{n,\, k}^j \quad \in \mathcal{D}(T^{n+1} \times S^1).
\]
\end{prop}

\begin{proof}
When $k = -1$, by definition and Corollary \ref{chicharacter}, $B_{n,\, -1}^0 = 1$ for $n \geq 0$. Thus we have $A_{n,\, -1}^0 = B_{n,\, -1}^0$.
Let $k \leq -2$. Since the both initial values $A_{0,\, k}^j$ and $B_{0,\, k}^j$ are equal to $t_1^{-j}$ for $0 \leq j \leq |k|-1$, in addition both satisfy the same recurrence relations (see Remark \ref{Arecurrence} and Lemma \ref{Brecurrence}), we have $A_{n,\, k}^j = B_{n,\, k}^j$ for $n \geq 0$ and $0 \leq j \leq |k|-1$.
\end{proof}

\begin{rem}
We can prove this proposition without using Lemma \ref{Brecurrence} as follows. Since we know that
$\lambda_{n+1}(k) J(\lambda_{n+1}(k)) = 1$ (see Remark \ref{lamandJ}), multiplying the both sides in (\ref{A+}) in Lemma \ref{A} by $\lambda_{n+1}(k)$, we obtain 
\begin{align*}
  A_{n,\, k}^j 
    &= \lambda_{n + 1}(k) \sum_{m \geq 0} \chi_{n,\, -km + j} \, t^m \\
    &= \left( \sum_{\alpha = 0}^{n + 1} (-1)^{\alpha} s_{\alpha}(t_1^k,\, \dots ,\, t_{n + 1}^k) t^{\alpha} \right) 
          \sum_{m \geq 0} \chi_{n,\, -km + j} \, t^m \\
    &= \sum_{l=0}^n \left( \sum_{\alpha = 0}^l (-1)^{\alpha} \chi_{n,\, (\alpha - l)k + j} \, 
    s_{\alpha} (t_1^k,\, \dots ,\, t_{n+1}^k) \right) t^l \\
    &= B_{n,\, k}^j.
\end{align*}
Here, we used Lemma \ref{chiandsym}. 
\end{rem}

\begin{prop}\label{B}
For $n\geq 0$,\, $k \in \mathbb{Z}$,\, $0 \leq j \leq \max\{0, \, |k|-1\}$, we have
\begin{enumerate}
  \item $\displaystyle \sum _{m \geq 0} \chi_{n,\, -km+j} \, t^m = B_{n,\, k}^j J_0 \left( \lambda_{n+1}(k)\right)$ , and

  \item $\displaystyle \sum _{m < 0} \chi_{n,\, -km+j} \, t^m = -B_{n,\, k}^j J_{\infty} \left( \lambda_{n+1}(k)\right)$ .
\end{enumerate}
In particular, the following equation holds:
\[
  \sum_{m \in \mathbb{Z}} \chi_{n,\, -km+j} \, t^m = B_{n, \, k}^j \, J\left( \lambda_{n+1}(k)\right) .
\]
\end{prop}

\begin{proof}
If $k < 0$, these immediately follow from Lemma \ref{A} and Proposition \ref{A=B}. It is sufficient to verify them for $k \geq 0$.
When $k = 0$, we can see them by the straightforward calculation. Since
\begin{align*}
  B_{n,\,0}^0 &= \sum_{l=0}^n \left( \sum_{i=1}^{n+1-l} (-1)^{i+l+1} \chi_{n,\, 0} \, s_{i+l}(t_1^0,\, \dots, \, t_{n+1}^0) \right) t^l \\
  &= \sum_{l=0}^n \left( \sum_{i=1}^{n+1-l} (-1)^{i+l+1} \Big( \substack{n+1 \\ \\ i+l} \Big) \right) t^l \\
  &= \sum_{l=0}^n \Big( \substack{n \\ \\ l} \Big) (-t)^l \\
  &= (1-t)^n ,
\end{align*}
we obtain 
\begin{align*}
  B_{n,\,0}^0 J_0(\lambda_{n+1}(0)) &= (1-t)^n \left[ \frac{1}{(1-t)^{n+1}} \right]^{t=0} = 
    \left[ \frac{1}{1-t} \right]^{t=0} = \sum_{m \geq 0} t^m \\
  -B_{n,\, 0}^0 J_{\infty}(\lambda_{n+1}(0)) &= (1-t)^n \left[ \frac{1}{(1-t)^{n+1}} \right]^{t= \infty} = 
    \left[ \frac{1}{1-t} \right]^{t= \infty} = \sum_{m<0} t^m.
\end{align*}
Let $k > 0$, and set $t=s^{-1}$, then from the case of $k < 0$,
\begin{align}
  \sum_{m>0} \chi_{n,\, -km+j} \, t^m &= \sum_{m>0} \chi_{n,\, -km+j} \, s^{-m} = 
    \sum_{\tilde{m} <0} \chi_{n,\, k\tilde{m} +j} \, s^{\tilde{m}} \notag \\
  &= -B_{n,\, -k}^j(t_1,\, \dots, \, t_{n+1}, s) \, J_{\infty}(\lambda_{n+1}(-k)) (t_1,\, \dots, \, t_{n+1}, s) . \label{eq1inB}
\end{align}
Since
\[
  \left[ \frac{1}{1-az} \right]^{z=\infty} = -a^{-1}z^{-1} \left[ \frac{1}{1-a^{-1}z^{-1}} \right]^{z^{-1}=0} 
    \quad ( a \in \mathbb{C} \setminus \{0\} ) ,
\]
we obtain
\begin{align}
  J_{\infty}(\lambda_{n+1}(-k)) (t_1,\, \dots, \, t_{n+1}, s) 
  &= \prod_{j=1}^{n+1} \left[ \frac{1}{1-t_j^{-k}s} \right]^{s =\infty} \notag \\
  &= (-1)^{n+1} t_1^k \dots t_{n+1}^k \, s^{-(n+1)} \prod_{j=1}^{n+1} \left[ \frac{1}{1-t_j^k s^{-1}} \right]^{s^{-1} = 0} \notag \\
  &= (-1)^{n+1} t_1^k \dots t_{n+1}^k \, t^{n+1} J_0(\lambda_{n+1}(k)) (t_1,\, \dots, \, t_{n+1}, t) . \label{eq2inB}
\end{align}
Moreover
\begin{align}
  & \quad -B_{n,\, -k}^j(t_1,\, \dots, \, t_{n+1}, s) \, (-1)^{n+1} t_1^k \dots t_{n+1}^k \, t^{n+1} \notag \\
  &= (-1)^n t_1^k \dots t_{n+1}^k \sum_{l=0}^n \left( \sum_{i=1}^{n+1-l} 
    (-1)^{i+l+1} \chi_{n,\, -ik+j} \, s_{i+l} (t_1^{-k},\, \dots, \, t_{n+1}^{-k}) \right) s^l \notag \\
  &= (-1)^n t_1^k \dots t_{n+1}^k \sum_{i=1}^{n+1} (-1)^{i+1} \chi_{n,\, -ik+j} \, s_i (t_1^{-k},\, \dots, \, t_{n+1}^{-k}) \notag \\
  & \qquad \qquad \qquad \qquad + (-1)^n \sum_{l=1}^n \left( \sum_{i=1}^{n+1-l} 
    (-1)^{i+l+1} \chi_{n,\, -ik+j} \, s_{n+1-i-l} (t_1^k,\, \dots, \, t_{n+1}^k) \right) t^{n+1-l} \notag \\
  &= \chi_{n,\, j} \sum_{\alpha = 0}^n (-1)^{\alpha} \, s_{\alpha}(t_1^k,\, \dots, \, t_{n+1}^k) \, t^{\alpha} 
    - \chi_{n,\, j} \lambda_{n+1}(k) 
    + \sum_{\alpha =1}^n \left( \sum_{i=1}^{\alpha} (-1)^{i+\alpha} \chi_{n,\, -ik+j} \, 
         s_{\alpha -i} (t_1^k,\, \dots, \, t_{n+1}^k) \right) t^{\alpha} \notag \\
  &= \chi_{n,\, j} + \sum_{\alpha =1}^n \left( (-1)^{\alpha} \chi_{n,\, j} s_{\alpha}(t_1^k,\, \dots, \, t_{n+1}^k) 
    + \sum_{i=1}^{\alpha} (-1)^{i+\alpha} \chi_{n,\, -ik+j} \, s_{\alpha -i} (t_1^k,\, \dots, \, t_{n+1}^k) \right) t^{\alpha} 
    - \chi_{n,\, j} \lambda_{n+1}(k) \notag \\
  &= \chi_{n,\, j} + \sum_{\alpha =1}^n \left( \sum_{i=0}^{\alpha} 
    (-1)^{i+\alpha} \chi_{n,\, -ik+j} \, s_{\alpha -i} (t_1^k,\, \dots, \, t_{n+1}^k) \right) t^{\alpha} 
    - \chi_{n,\, j} \lambda_{n+1}(k) \notag \\
  &= \chi_{n,\, j} + \sum_{\alpha =1}^n \left( \sum_{\beta =0}^{\alpha} 
    (-1)^{\beta} \chi_{n,\, (\beta - \alpha )k + j} \, s_{\beta} (t_1^k,\, \dots, \, t_{n+1}^k) \right) t^{\alpha} 
    - \chi_{n,\, j} \lambda_{n+1}(k) \notag \\
  &= \sum_{\alpha =0}^n \left( \sum_{\beta =0}^{\alpha} 
    (-1)^{\beta} \chi_{n,\, (\beta - \alpha )k + j} \, s_{\beta} (t_1^k,\, \dots, \, t_{n+1}^k) \right) t^{\alpha} 
    - \chi_{n,\, j} \lambda_{n+1}(k) \notag \\
  &= B_{n,\, k}^j (t_1,\, \dots, \, t_{n+1}, t) - \chi_{n,\, j} \lambda_{n+1}(k) . \label{eq3inB}
\end{align}
Thus from (\ref{eq1inB}), (\ref{eq2inB}) and (\ref{eq3inB}), we have
\begin{align*}
  \sum_{m>0} \chi_{n,\, -km+j} \, t^m 
  &= -B_{n,\, -k}^j(t_1,\, \dots, \, t_{n+1}, s) \, J_{\infty}(\lambda_{n+1}(-k)) (t_1,\, \dots, \, t_{n+1}, s) \\
  &= -B_{n,\, -k}^j(t_1,\, \dots, \, t_{n+1}, s) \, (-1)^{n+1} t_1^k \dots t_{n+1}^k \, t^{n+1} \, 
    J_0(\lambda_{n+1}(k)) (t_1,\, \dots, \, t_{n+1}, t) \\
  &= \left( B_{n,\, k}^j (t_1,\, \dots, \, t_{n+1}, t) - \chi_{n,\, j} \lambda_{n+1}(k) \right) \, 
    J_0(\lambda_{n+1}(k)) (t_1,\, \dots, \, t_{n+1}, t) \\
  &= B_{n,\, k}^j (t_1,\, \dots, \, t_{n+1}, t) \, J_0(\lambda_{n+1}(k)) (t_1,\, \dots, \, t_{n+1}, t) - \chi_{n,\, j}.
\end{align*}


Similarly using the relation
\[
  \left[ \frac{1}{1-az} \right]^{z=0} = -a^{-1}z^{-1} \left[ \frac{1}{1-a^{-1}z^{-1}} \right]^{z^{-1}=\infty} 
    \quad (a \in \mathbb{C} \setminus \{0\})
\]
and (\ref{eq3inB}), we obtain the following:
\begin{align*}
  \sum_{m \leq 0} \chi_{n,\, -km+j} \, t^m &= \sum_{\tilde{m} \geq 0} \chi_{n,\, k\tilde{m} +j} \, s^{\tilde{m}} \\
  &= B_{n,\, -k}^j(t_1,\, \dots, \, t_{n+1}, s) \, J_0(\lambda_{n+1}(-k)) (t_1,\, \dots, \, t_{n+1}, s) \\
  &= B_{n,\, -k}^j(t_1,\, \dots, \, t_{n+1}, s) \, (-1)^{n+1} t_1^k \dots t_{n+1}^k \, t^{n+1} \, 
    J_{\infty}(\lambda_{n+1}(k)) (t_1,\, \dots, \, t_{n+1}, t) \\
  &= \left( -B_{n,\, k}^j (t_1,\, \dots, \, t_{n+1}, t) + \chi_{n,\, j} \lambda_{n+1}(k) \right) \, 
    J_{\infty}(\lambda_{n+1}(k)) (t_1,\, \dots, \, t_{n+1}, t) \\
  &= -B_{n,\, k}^j (t_1,\, \dots, \, t_{n+1}, t) \, J_{\infty}(\lambda_{n+1}(k)) (t_1,\, \dots, \, t_{n+1}, t) + \chi_{n,\, j}.
\end{align*}
Thus we get the required equations.
\end{proof}


Our main theorem immediately follows from the above proposition.

\begin{proof}[Proof of Theorem \ref{mainthm}]
Let $n \geq 1$,\, $k \in \mathbb{Z}$. Then, from (\ref{key3}) and Proposition \ref{B}, we verify the following:
\begin{align*}
  \ind_{T^{n+1} \times S^1} \left( p_k^* [D_n] \right)
  &= \sum_{m \in \mathbb{Z}} \chi_{n,\, -km} \, t^m \\
  &= \left( \sum_{l=0}^n \left( \sum_{i=1}^{n+1-l} 
    (-1)^{i+l+1} \chi_{n,\, ik} \, s_{i+l}(t_1^k,\,\dots, \, t_{n+1}^k)\right) t^l \right)J(\lambda_{n+1}(k)) \\
  &= \left( \sum_{l=0}^n \left( \sum_{\alpha =0}^{l} 
    (-1)^{\alpha} \chi_{n,\, (\alpha -l)k} \, s_{\alpha}(t_1^k,\,\dots, \, t_{n+1}^k)\right) t^l \right)J(\lambda_{n+1}(k)) .
\end{align*}
The last equality follows from Remark \ref{anoB}.
\end{proof}

%
%

\section{Application}

In this section, let us give an application of our main theorem.
Firstly, recall Lefschetz formula (see \cite{lef1} for details):

\begin{thm}[Lefschetz formula for group actions]
Let $M$ be a compact manifold on which a compact Lie group $G$ acts smoothly. For a $G$-invariant elliptic operator $A \colon \Gamma(M ,\, \mathcal{E}^+) \to \Gamma(M ,\, \mathcal{E}^-)$, the following equation holds
\[
  \ind_G \left( [ \sigma(A) ] \right) (g) = \sum_{p \in M^g} 
    \frac{ \chi \left( \mathcal{E}^+|_{p} \right)(g) - \chi \left( \mathcal{E}^-|_{p} \right)(g) }{ \text{det} 
      \left( 1 - g|_{T_pM} \right) } \qquad \in \mathbb{C}
\]
for $g \in G$ such that the fixed point set $M^g := \{ x \in M \left| \, g \cdot x = x \right. \}$ is finite.
\end{thm}

\vspace{3mm}

In particular, if we consider the Dolbeault complex, the above formula can be deformed as follows: 
Let $M$ be a compact complex manifold provided with a holomorphic action of a compact Lie group $G$ and $\displaystyle \mathcal{E}^+ := \bigoplus_{q : even} \bigwedge\nolimits^q \overline{T_{\mathbb{C}}^*M}$, $\displaystyle \mathcal{E}^- := \bigoplus_{q : odd} \bigwedge\nolimits^q \overline{T_{\mathbb{C}}^*M}$ be two $G$-equivariant vector bundles over $M$, where $T_{\mathbb{C}}^*M$ is the holomorphic cotangent bundle of $M$.
We assume that the Dirac operator $D := \bar{\partial} + \bar{\partial}^*$ associated with the Dolbeault operator on $M$ is $G$-invariant. In addition, let $[ E^0 ] - [ E^1 ] \in K_{G}(M)$.
Then since we have that
\[
  \chi \left( \mathcal{E}^+|_{p} \otimes E^i|_p \right)(g) - \chi \left( \mathcal{E}^-|_{p} \otimes E^i|_p  \right)(g) = 
    \chi \left( E^i|_p \right)(g) \; \text{det}_{\mathbb{C}} (1 - g|_{\overline{T_{\mathbb{C}}^*M} |_p }) 
\]
and
\[
  \text{det}_{\mathbb{R}} ( 1 - g|_{ TM |_p} ) = \text{det}_{\mathbb{C}} ( 1 - g|_{T_{\mathbb{C}} M |_p} ) 
    \text{det}_{\mathbb{C}} ( 1 - g|_{ \overline{T_{\mathbb{C}} M} |_p} )
\]
for each fixed point $p \in M^g$,
it follows from Lefschetz formula that
\begin{align*}
  \ind_G \left( [ D ] \otimes \left( [ E^0 ] - [ E^1 ] \right) \right) (g)
  &= \ind_G \left( [ D ] \otimes [ E^0 ] \right) (g) - \ind_G \left( [ D ] \otimes [ E^1 ] \right) (g) \\
  &= \sum_{p \in M^g} \frac{ \chi \left( E^0|_p \right)(g) - \chi \left( E^1|_p \right)(g) }{ 
    \text{det}_{\mathbb{C}} ( 1 - g|_{ \overline{T_{\mathbb{C}} M} |_p} )}
\end{align*}
for $g \in G$ such that $M^g$ is finite.

As a special case, we consider $\mathbb{CP}^n$ on which $T^{n+1}$ acts canonically (as in the previous sections) and the Dirac operator $D_n$ associated with the Dolbeault operator.
Then we have the following:

For $(t_1 ,\, \dots ,\, t_{n+1}) \in T^{n+1}$ such that $t_i \neq t_j (i \neq j)$,  
\begin{equation} \label{lefcp}
  \ind_{T^{n+1}} \left( [ D_n ] \otimes \left( [E^0] - [E^1] \right) \right) (t_1,\, \dots ,\, t_{n + 1}) = 
    \sum_{i=1}^{n+1} \frac{ \chi ({E^0|_{p_i}} ) - \chi( {E^1|_{p_i}} )}{\displaystyle \prod_{j \neq i} \left( 1-t_j^{-1}t_i \right)} 
      \qquad \in \mathbb{C}
\end{equation}
for any $T^{n+1}$-equivariant complex vector bundles $E^{i} (i =0,\,1)$ over $\mathbb{CP}^n$, where $p_i$'s are the fixed points $[ 0 : \dots : 1 : \dots : 0] \in \mathbb{CP}^n (1 \leq i \leq n+1)$.

\vspace{0.5cm}

We derive the above equation $(\ref{lefcp})$ by using our main theorem.
In this section, we regard the transversal index of an element in $K_{T^{n+1} \times S^1}(T_{S^1}^* S(\mathcal{O}(k)))$ as a continuous linear map $\mathcal{D}(S^1) \to \mathcal{D}(T^{n+1})$ (see Remark \ref{spform}).
Analogously, $J_{0}(\lambda_{n+1}(k))$, $J_{\infty}(\lambda_{n+1}(k))$ and $J(\lambda_{n+1}(k))$ are regarded as linear maps $\mathcal{D}(S^1) \to \mathcal{D}(T^{n+1})$ (see Remark \ref{lamandJ}).

We denote by $\homo^{C^0}\left( \mathcal{D}(S^1) ,\, \mathcal{D}(T^{n+1}) \right)$ the space of continuous linear maps $\mathcal{D}(S^1) \to \mathcal{D}(T^{n+1})$.
We define an action of the representation ring $R(T^{n+1} \times S^1)$ on $\homo^{C^0}\left( \mathcal{D}(S^1) ,\, \mathcal{D}(T^{n+1}) \right)$ as follows:
\[
  \left( \left( fg \right) \Phi \right) (\varphi) := f \Phi \left( g \varphi \right) \qquad (\varphi \in \mathcal{D}(S^1))
\]
for $f \in R(T^{n+1})$, $g \in R(S^1)$ and $\Phi \in \homo^{C^0}\left( \mathcal{D}(S^1) ,\, \mathcal{D}(T^{n+1}) \right)$.
Then it turns out that 
\[
  \ind_{T^{n+1} \times S^1} \colon K_{T^{n+1} \times S^1}(T_{S^1}^* S(\mathcal{O}(k))) \to 
    \homo^{C^0}\left( \mathcal{D}(S^1) ,\, \mathcal{D}(T^{n+1}) \right)
\]
is an $R(T^{n+1} \times S^1)$-module homomorphism (see Remark \ref{module} and \ref{spform}).

\begin{lem} \label{residue}
For any analytic function $\varphi$ on $S^1$, 
\[
  J(\lambda_{n+1}(k))(\varphi) = 
    - \sum_{t = \xi \in S^1} \res \left( \frac{\varphi(t)}{\lambda_{n+1}(k)} \frac{1}{t} \;\; ; \, \xi \right) 
    \qquad \in \mathcal{D}(T^{n+1}) .
\]
Additionally, for any $f = f(t_1,\, \dots ,\, t_{n+1} ,\, t) \in R(T^{n + 1} \times S^1)$,
\[
  \left( f J(\lambda_{n+1}(k)) \right) (\varphi) = - \sum_{t = \xi \in S^1} 
    \res \left( \frac{ f(t_1,\, \dots ,\, t_{n+1} ,\, t) \varphi(t)}{\lambda_{n+1}(k)} \frac{1}{t} \;\; ; \, \xi \right) 
      \qquad \in \mathcal{D}(T^{n+1}) ,
\]
where $\res (\phi \; ; \xi)$ is the residue of $\phi$ at $\xi$.
\end{lem}

\begin{proof}
By definition, $J(\lambda_{n+1}(k)) = J_{0}(\lambda_{n+1}(k)) - J_{\infty}(\lambda_{n+1}(k))$ and $J_{0}(\lambda_{n+1}(k))$ , $J_{\infty}(\lambda_{n+1}(k))$ are the Laurent expansions of $\frac{1}{\lambda_{n+1}(k)}$ at $t = 0$ and $t = \infty$, so these can be written by the following forms:
\begin{align*}
  J_{0}(\lambda_{n+1}(k)) = \sum_{m \in \mathbb{Z}} c_m^0 t^m , &\qquad c_m^0 := 
    \frac{1}{2\pi i} \int_{\Gamma_{1 - \epsilon}} \frac{t^{-m}}{\lambda_{n+1}(k)} \frac{dt}{t} , \\
  J_{\infty}(\lambda_{n+1}(k)) = \sum_{m \in \mathbb{Z}} c_m^{\infty} t^m , &\qquad c_m^{\infty} := 
    \frac{1}{2\pi i} \int_{\Gamma_{1 + \epsilon}} \frac{t^{-m}}{\lambda_{n+1}(k)} \frac{dt}{t} ,
\end{align*}
where $0 < \epsilon < 1$ and $\Gamma_{r}$ is the circle centered at the origin with radius $r >0$.
Since $\varphi$ is analytic, if we write $\displaystyle \varphi = \sum_{m \in \mathbb{Z}} \varphi_m t^m$ using the Fourier expansion, then 
\begin{align*}
  J(\lambda_{n+1}(k)) (\varphi) &= \sum_{m \in \mathbb{Z}} \left( c_{-m}^0 - c_{-m}^{\infty} \right) \varphi_m \\
  &= \sum_{m \in \mathbb{Z}} \frac{1}{2\pi i} \int_{\Gamma_{1 - \epsilon} - \Gamma_{1 + \epsilon}} 
    \frac{\varphi_m t^m}{\lambda_{n+1}(k)} \frac{dt}{t} \\
  &= \frac{1}{2\pi i} \int_{\Gamma_{1 - \epsilon} - \Gamma_{1 + \epsilon}} \frac{\varphi(t)}{\lambda_{n+1}(k)} \frac{dt}{t} \\
  &= - \sum_{t = \xi \in S^1} \res \left( \frac{\varphi(t)}{\lambda_{n+1}(k)} \frac{1}{t}  \;\;  ; \; \xi \right) .
\end{align*}
Moreover any element in $R(T^{n+1} \times S^1)$ can be written in the form $\displaystyle f = \sum_{l \in \mathbb{Z}} a_l t^l$ $( a_l \in R(T^{n+1}) )$, where $a_l = 0$ for all but finitely many $l \in \mathbb{Z}$. Then by definition and the above computation, we have
\begin{align*}
  \left( f J(\lambda_{n+1}(k)) \right) (\varphi) &= \sum_{l \in \mathbb{Z}} a_l \, J(\lambda_{n+1}(k))  (t^l \varphi) \\
  &= -\sum_{l \in \mathbb{Z}} a_l \sum_{t = \xi \in S^1} 
    \res \left( \frac{ t^l \varphi(t)}{\lambda_{n+1}(k)} \frac{1}{t}  \;\;  ; \; \xi \right) \\
  &= - \sum_{t = \xi \in S^1} \res \left( \frac{ 
    \left( \sum_{l} a_l t^l \right) \, \varphi(t)}{\lambda_{n+1}(k)} \frac{1}{t}  \;\;  ; \; \xi \right) .
\end{align*}
\end{proof}

From now on, we set $k = -1$. 
We simply denote by $p$ the projection $p_{-1} \colon S(\mathcal{O}(-1)) \to \mathbb{CP}^n$.
Note that since $\mathcal{O}(-1)$ is a subbundle of $\mathbb{CP}^n \times \mathbb{C}^{n+1}$, the projection $\mathbb{CP}^n \times \mathbb{C}^{n+1} \to \mathbb{C}^{n+1}$ induces the natural $T^{n+1} \times S^1$-equivariant isomorphism $S(\mathcal{O}(-1)) \xrightarrow{\cong} S^{2n+1}$ as follows:
\[
  \left( [ z_1 : \dots : z_{n+1} ] ,\, (v_1 ,\, \dots ,\, v_{n+1}) \right) \mapsto \left( v_1 ,\, \dots ,\, v_{n+1} \right) .
\]
Recall that the action of $T^{n+1} \times S^1$ on $S(\mathcal{O}(-1)) \cong S^{2n+1} \subset \mathbb{C}^{n+1}$ is given by 
\[
  (z_1 ,\, \dots ,\, z_{n+1}) \mapsto (u_1u^{-1} z_1 ,\, \dots ,\, u_{n+1}u^{-1} z_{n+1})
\]
for $(z_1 ,\, \dots ,\, z_{n+1}) \in S^{2n+1}$ and $(u_1 ,\, \dots ,\, u_{n+1} ,\, u) \in T^{n+1} \times S^1$.

Next we compute $K$-groups. See \cite{Ktheory}, \cite{karoubi}.

\begin{lem} \label{ringisom}
There exist the following ring isomorphisms:
\begin{enumerate}
  \item \label{free}
    $\begin{aligned}[t]
      p^* \colon K_{T^{n+1}}(\mathbb{CP}^n) \xrightarrow{\cong} K_{T^{n+1} \times S^1}(S(\mathcal{O}(-1))) \quad \text{and}
    \end{aligned}$
  \item
    $\begin{aligned}[t] \label{repandK}
      \beta_0 \colon R(T^{n+1} \times S^1) / ( \lambda_{n+1}(-1) ) \xrightarrow{\cong} 
        K_{T^{n+1} \times S^1}(S(\mathcal{O}(-1))) ,
    \end{aligned}$
\end{enumerate}
where $( \lambda_{n+1}(-1) )$ is an ideal of $R(T^{n+1} \times S^1)$ generated by $\lambda_{n+1}(-1)$.
\end{lem}

\begin{proof}
Since the action of $S^1$ on $S(\mathcal{O}(-1))$ is free, there exists the above isomorphism.
More precisely, the pull-back of $p$ induces the above ring homomorphism and its inverse map is given by taking the quotient by the action of $S^1$.

Let us regard $V := \mathbb{C}^{n+1}$ a representation space of $T^{n+1} \times S^1$ on which the action is the natural extension of one on $S^{2n+1} = S(\mathcal{O}(-1))$, and $B^{2n+2}$ be the closed unit ball of $V$.
We consider the following exact sequence for the pair $(B^{2n+2} ,\, S^{2n+1})$: 
\[
  \begin{CD}
    K_{T^{n+1} \times S^1}^0 (B^{2n+2} ,\, S^{2n+1})    @>>>   K_{T^{n+1} \times S^1}^0 (B^{2n+2})   
      @>\tilde{\beta}_0>>   K_{T^{n+1} \times S^1}^0 (S^{2n+1})   \\
    @AAA       @.     @VVV   \\
    K_{T^{n+1} \times S^1}^1 (S^{2n+1})   @<<<   K_{T^{n+1} \times S^1}^1 (B^{2n+2})   @<<<   
      K_{T^{n+1} \times S^1}^1 (B^{2n+2} ,\, S^{2n+1})
  \end{CD}
\]
Here, since $B^{2n+2}$ is $T^{n + 1}$-equivariant contractible, we have that 
\[
  K_{T^{n+1} \times S^1}^0 (B^{2n+2}) \cong K_{T^{n+1} \times S^1}^0 (pt) \cong R(T^{n+1} \times S^1) .
\]
In addition, using Thom isomorphim, we obtain the followings
\begin{align*}
  &K_{T^{n+1} \times S^1}^1 (B^{2n+2} ,\, S^{2n+1}) \cong K_{T^{n+1} \times S^1}^1 ( V ) \cong 
    K_{T^{n+1} \times S^1}^1 ( pt ) \cong 0 , \\[7pt]
  &K_{T^{n+1} \times S^1}^0 (B^{2n+2} ,\, S^{2n+1}) \cong K_{T^{n+1} \times S^1}^0 ( V ) \cong 
    K_{T^{n+1} \times S^1}^0 (pt) \cdot \lambda_{V} \cong R(T^{n+1} \times S^1) \lambda_{V} ,
\end{align*}
where $\lambda_{V} \in K_{T^{n+1} \times S^1}^0 (V)$ is the Thom class of $V \to pt$. Then, since the ristriction of $\lambda_{V}$ to $pt$ is $\displaystyle \sum_{i = 0}^{n + 1} \left( -1 \right)^i \left[ \bigwedge\nolimits^i V^* \right] \in K_{T^{n + 1}}(pt)$, its character is equal to $\lambda_{n+1}(-1)$.
Thus $\tilde{\beta}_0$ induces the required ring isomorphism $\beta_0$.
\end{proof}

\begin{rem} \label{beta}
\begin{enumerate}
  \item \label{RImodule}
It immediately follows from this Lemma that $\lambda_{n+1}(-1) \cdot K_{T^{n+1} \times S^1} \left( S(\mathcal{O}(-1)) \right) = 0$. Moreover since $K_{T^{n+1} \times S^1} (T_{S^1}^* S(\mathcal{O}(-1)))$ is a module over $K_{T^{n+1} \times S^1} \left( S(\mathcal{O}(-1)) \right)$, we have $\lambda_{n+1}(-1) \cdot K_{T^{n+1} \times S^1} (T_{S^1}^* S(\mathcal{O}(-1))) = 0$.
Then it turns out that the map 
\[
  \ind_{T^{n+1} \times S^1} \colon K_{T^{n+1} \times S^1} (T_{S^1}^* S(\mathcal{O}(-1))) \to N
\]
is a $R(T^{n+1} \times S^1) / ( \lambda_{n+1}(-1) )$-module homomorphism, where $N$ is the image of $\ind_{T^{n+1} \times S^1}$.
  \item We denote by $\beta$ the isomorphism given by the composition
\[
  \beta := \beta_0 \circ \left (p^* \right)^{-1} \colon R(T^{n+1} \times S^1) / ( \lambda_{n+1}(-1) ) 
    \xrightarrow{\cong} K_{T^{n+1}}(\mathbb{CP}^n)
\]
and also denote by $\alpha$ the inverse map of $\beta$:
\[
  \alpha := \beta^{-1} \colon K_{T^{n+1}}(\mathbb{CP}^n) \xrightarrow{\cong} R(T^{n+1} \times S^1) / ( \lambda_{n+1}(-1) ) .
\]
From the above proof, we can describe the isomorphism $\beta$ explicitly as follows:
For $[ V^0 ] - [ V^1 ] \bmod \lambda_{n+1}(-1)$,
\[
  \beta \left( [ V^0 ] - [ V^1 ] \bmod \lambda_{n+1}(-1) \right) = 
    \left[ S^{2n+1} \underset{S^1}{\times} V^0 \right] - \left[ S^{2n+1} \underset{S^1}{\times} V^1 \right] 
      \quad \in K_{T^{n+1}}(\mathbb{CP}^n) ,
\]
where $S^{2n+1} \underset{S^1}{\times} V^j := \left( S^{2n+1} \times V^j \right) / \sim$ and the equivalent relation is given by
\[
  ( v ,\, \xi ) \sim ( u^{-1} v ,\, u \cdot \xi )  \quad \text{for some} \quad u \in S^1 .
\]
\end{enumerate}
\end{rem}

Let $1 \leq i \leq n + 1$.
If we denote by $\iota_{p_i} \colon \left\{ p_i \right\} = \left\{ [ 0 : \dots : 1 : \dots : 0] \right\} \hookrightarrow \mathbb{CP}^n$ the inclusion map, then the following map
\[
  K_{T^{n+1}}(\mathbb{CP}^n) \xrightarrow{\iota_{p_i}^* } K_{T^{n+1}}(p_i) = R(T^{n+1}) 
    \xrightarrow[character]{\chi} \mathcal{D}(T^{n+1}) 
\]
is given by $[ E^0 ] - [ E^1 ] \mapsto \chi(E^0|_{p_i}) - \chi(E^1|_{p_i})$.

On the other hand, 
since we identify $R(T^{n + 1} \times S^1)$ and $R(T^{n + 1})$ with $\mathbb{Z} \left[ t_1^{\pm} ,\, \dots ,\, t_{n + 1}^{\pm} ,\, t^{\pm} \right]$ and $\mathbb{Z} \left[ t_1^{\pm} ,\, \dots ,\, t_{n + 1}^{\pm} \right]$ respectively,
the substitution $t = t_i$ induces the ring homomorphism 
\[
\cdot \; |_{t = t_i} \colon R(T^{n+1} \times S^1) \to R(T^{n + 1}) .
\]
In addition, this lifts to the ring homomorphism
\[
  \mu_i \colon R(T^{n+1} \times S^1) / ( \lambda_{n+1}(-1) ) \to R(T^{n + 1})
\]
because $\lambda_{n+1}(-1) |_{t = t_i} = 0$.  $\mu_i$ is also denoted by $\cdot \; |_{t = t_i}$.

Then we can identify $\iota_{p_i}^*$ with $\mu_i$ via the above ring isomorphism $\alpha$, that is,

\begin{lem} \label{fiberonfixpt}
There exists a commutative diagram of rings:
\[
  \begin{CD}
    K_{T^{n+1}}(\mathbb{CP}^n) @>{\alpha}>{\cong}> R(T^{n+1} \times S^1) / ( \lambda_{n+1}(-1) ) \\
    @VV{\iota_{p_i}^*}V                            @VV{\cdot |_{t = t_i}}V                               \\
    K_{T^{n + 1}}(p_i)              @=              R(T^{n + 1})
  \end{CD}
\]

In other words,
for any $[ E^0 ] - [ E^1 ] \in K_{T^{n+1}}(\mathbb{CP}^n)$,  the following equation holds:
\[
  \chi( E^0|_{p_i} ) - \chi( E^1|_{p_i} ) = \alpha ( [ E^0 ] ) |_{t = t_i} - \alpha ( [ E^1 ] ) |_{t = t_i} 
\]
in $R(T^{n+1}) \cong \mathbb{Z} \left[ t_1^{\pm 1} ,\, \dots ,\, t_{n+1}^{\pm 1} \right]$ for $1 \leq i \leq n+1$.
\end{lem}


\begin{proof}
If we write 
\begin{equation} \label{alpha1}
  \alpha ( [ E^0 ] - [ E^1 ] ) = [ V^0 ] - [ V^1 ] \bmod \lambda_{n+1}(-1)
\end{equation}
for some finite dimensional representation spaces $V^0$, $V^1$ of $T^{n+1} \times S^1$, 
by the definition of $\mu_i = \cdot |_{t = t_i}$,
we have
\[
  \alpha ( [ E^0 ] ) |_{t = t_i} - \alpha ( [ E^1 ] ) |_{t = t_i} = \chi (V^0)|_{t = t_i} - \chi (V^1)|_{t = t_i} .
\]
Thus it is sufficient to verify the following:
\begin{equation} \label{minigoal}
  \chi \left( E^0|_{p_i} \right) - \chi \left( E^1|_{p_i} \right) = \chi (V^0)|_{t = t_i} - \chi (V^1)|_{t = t_i} .
\end{equation}
From (\ref{alpha1}),
\[
  [ E^0 ] - [ E^1 ] = \beta \left( [ V^0 ] - [ V^1 ] \bmod \lambda_{n+1}(-1) \right) = 
    \left[ S^{2n+1} \underset{S^1}{\times} V^0 \right] - \left[ S^{2n+1} \underset{S^1}{\times} V^1 \right]
\]
in $K_{T^{n+1}}(\mathbb{CP}^n)$.
Let $\displaystyle V^j = \bigoplus_{l \in \mathbb{Z}} W_l^j \otimes t^l$ $(j = 1,\, 2)$, where $W_l^j$ are finite dimensional representation spaces of $T^{n+1}$. As $V^j$ are finite, $W_l^j = 0$ for all but finitely many $l \in \mathbb{Z}$.
Then $S^{2n+1} \underset{S^1}{\times} V^j$ can be written in the following form:
\[
  S^{2n+1} \underset{S^1}{\times} V^j \cong \bigoplus_{l} \mathcal{W}_l^j \otimes \mathcal{L}(l) ,
\]
where $\mathcal{W}_l^j$ are trivial vector bundles over $\mathbb{CP}^n$ whose fibers are $W_l^j$, and $\mathcal{L}(l)$ are complex line bundles defined by $S^{2n+1} \underset{S^1}{\times} t^l$. Note that since the actions of $S^1$ on $S^{2n+1} \times t^l$ are given by
\[
  S^1 \ni a \colon (v ,\, \xi ) \mapsto (a^{-1}v ,\, a^l \xi ) \qquad (v ,\, \xi) \in S^{2n+1} \times t^l,
\]
each $\mathcal{L}(l)$ is isomorphic with $\mathcal{O}(l)$ as line bundles. However the actions of $T^{n+1}$ on fibers over fixed points $p_i$ are different from each other. Indeed, if we use the following local trivializations of $\mathcal{L}(l)$:
\[
  \left. \left( S^{2n+1} \times t^l \right) \right|_{U_i} \ni [ v ,\, \xi ] \longmapsto 
    \left( [ v ] ,\, \left( \frac{ v_i }{ |v_i| } \right)^l \xi \right) \in U_i \times \mathbb{C} ,
\]
where each $U_i$ is an open set of $\mathbb{CP}^n$ used in the proof of Lemma \ref{cohomologyn} and $v = ( v_1,\, \dots v_{n+1} ) \in S^{2n+1}$,
we obtain that $\chi (\mathcal{L}(l) |_{p_i}) = t_i^l$ (incidentally, $\chi (\mathcal{O}(l)|_{p_i} ) = t_i^{-l} $). 
Therefore we have (\ref{minigoal}) as follows:
\begin{align*}
  \chi \left( E^0|_{p_i} \right) - \chi \left( E^1|_{p_i} \right) 
  &= \chi \left( \left. \left( S^{2n+1} \underset{S^1}{\times} V^0 \right) \right|_{p_i} \right) 
    - \chi \left( \left. \left( S^{2n+1} \underset{S^1}{\times} V^1 \right) \right|_{p_i} \right) \\
  &= \chi \left( \left. \left( \bigoplus_{l} \mathcal{W}_l^0 \otimes \mathcal{L}(l) \right) \right|_{p_i} \right) 
    - \chi \left( \left. \left( \bigoplus_{l} \mathcal{W}_l^1 \otimes \mathcal{L}(l) \right) \right|_{p_i} \right) \\[3pt]
  &= \sum_{l} \chi(W_l^0) t_i^l - \sum_{l} \chi(W_l^1) t_i^l \\
  &= \chi(V^0)|_{t = t_i} - \chi(V^1)|_{t = t_i} 
\end{align*}
for $1 \leq i \leq n+1$.
\end{proof}

Finally, we prove Corollary \ref{application}.

\begin{proof}[Proof of Corollary \ref{application}]
From Lemma \ref{ringisom} and Remark \ref{beta} (\ref{RImodule}), for any $\eta = [ E^0 ] - [ E^1 ] \in K_{T^{n+1}}(\mathbb{CP}^n)$, we have
\begin{equation} \label{diag}
  \ind_{T^{n+1} \times S^1} \left( p^*[ D_n ] \otimes p^*\eta \right) = 
    \alpha \left( \eta \right) \ind_{T^{n+1} \times S^1} \left( p^*[ D_n ] \right)
\end{equation}
as linear maps $\mathcal{D}(S^1) \to \mathcal{D}(T^{n+1})$.
Moreover if we denote by $[ D_n^{\eta} ]$ the class of $[ D_n ] \otimes \eta$ in $K_{T^{n+1}}(T^* \mathbb{CP}^n)$, then $p^*[ D_n ] \otimes p^*\eta = p^*[ D_n^{\eta} ]$. Thus from Corollary \ref{spfap}, we obtain that
\[
  \ind_{T^{n+1} \times S^1} \left( p^*[ D_n^{\eta} ] \right) = 
    \sum_{m \in \mathbb{Z}} \ind_{T^{n+1}} \left( [ D_n^{\eta} ] \otimes \underline{t}^{-m} \right) t^m
\]
and from Theorem \ref{mainthm}, we obtain that
\[
  \ind_{T^{n+1} \times S^1} \left( p^*[ D_n ] \right) = J(\lambda_{n+1}(-1)) .
\]
Thus substituting the constant function $1$ on $S^1$ for (\ref{diag}) and using Lemma \ref{residue}, we have
\[
  \ind_{T^{n+1}} \left( [ D_n ] \otimes \eta \right) = 
    - \sum_{t = \xi \in S^1} \res \left( \frac{\alpha ( \eta ) }{\lambda_{n+1}(-1)} \frac{1}{t} \, ; \, \xi \right) .
\]
If $t_i \neq t_j$ for $i \neq j$, the residues in the above are computable explicitly. Indeed, since $\displaystyle \lambda_{n+1}(-1) = \prod_{j = 1}^{n+1} ( 1 - t_j^{-1} t)$ and $\alpha(\eta)$ is a Laurent polynomial in variables $t_1,\, \dots ,\, t_{n+1}$ and $t$, we conclude that $\frac{\alpha ( \eta ) }{\lambda_{n+1}(-1)} \frac{1}{t}$ does not have poles anywhere except at $t = t_i (i = 1,\, \dots ,\, n+1)$ and the orders at these points are at most $1$. Then using Lemma \ref{fiberonfixpt}, we have
\[
  - \res \left( \frac{\alpha ( [ E^0 ]) - \alpha ( [ E^1 ] ) }{\lambda_{n+1}(-1)} \frac{1}{t} \;\; ; \; t_i \right) = 
    \frac{ \chi ( E^0|_{p_i} ) - \chi ( E^1|_{p_i} )}{\displaystyle \prod_{j \neq i} \left( 1 - t_j^{-1} t_i \right)}
\]
for $i = 1,\, \dots ,\, n+1$.
\end{proof}

\vspace{1cm}

Department of Mathematics, Faculty of Science, Kyoto University

Sakyo-ku, Kyoto 606-8502, JAPAN

miseki@math.kyoto-u.ac.jp

\end{document}